\documentclass[ejs]{imsart}
\usepackage{graphicx}
\usepackage{caption}
\usepackage{subcaption}
\RequirePackage[colorlinks,citecolor=blue,urlcolor=blue]{hyperref}

\usepackage[section]{placeins}
\usepackage[toc,page]{appendix}

\usepackage{enumitem}
\usepackage{color}
\usepackage{colordvi}
\usepackage{mathrsfs}
\usepackage{amsmath,amsfonts,amsthm, amssymb}

\usepackage{lineno}

\def\var{\mathop{\rm Var}\nolimits}
\def\cov{\mathop{\rm Cov}\nolimits}
\def\sd{\mathop{\rm sd}\nolimits}
\def\II{1}
\def\Cref{\ref}
\def\PY{\mathop{\rm PY}\nolimits}
\def\CPY{\mathop{\rm CPY}\nolimits}

\mathchardef\given="626A
\def\X{{\cal X}}
\def\A{{\cal A}}
\def\F{{\cal F}}
\def\G{{\cal G}}
\def\E{\EE}
\def\FF{\mathbb{F}}
\def\GG{\mathbb{G}}
\def\PP{\mathbb{P}}
\def\EE{\mathbb{E}}
\def\NN{\mathbb{N}}
\def\RR{\mathbb{R}}
\def\WW{\mathbb{W}}

\def\ra{\rightarrow}
\def\s{\sigma}
\def\a{\alpha}
\def\b{\beta}
\def\d{\delta}
\def\eps{\epsilon}
\def\e{\epsilon}
\def\l{\lambda}
\def\g{\gamma}
\def\Beta{\mathop{\rm B}\nolimits}
\def\Dir{\mathop{\rm Dir}\nolimits}
\let\weak\rightsquigarrow
\def\linfty{\ell^\infty}
\def\sumin{\sum_{i=1}^n}
\newcommand{\ind}{\, \raise-2pt\hbox{$\stackrel{\mbox{\scriptsize ind}}{\sim}$}\, }
\newcommand{\iid}{\, \raise-2pt\hbox{$\stackrel{\mbox{\scriptsize iid}}{\sim}$}\,}

\theoremstyle{plain}
\newtheorem{theorem}{Theorem}

\newtheorem{lemma}[theorem]{Lemma}

\newtheorem{corollary}[theorem]{Corollary}
\newtheorem*{example}{Example}

\begin{document}

\begin{frontmatter}
\title{Bernstein-von Mises theorem for the Pitman-Yor process of nonnegative type}
\runtitle{BvM for the Pitman-Yor process}
\maketitle

\begin{aug}
\author{\fnms{S.E.M.P.} \snm{Franssen\thanksref{t1}}\ead[label=e1]{s.e.m.p.franssen@tudelft.nl}}
\address{TU Delft \\ \printead{e1}}
\author{\fnms{A.W.} \snm{van der Vaart\thanksref{t1}}\ead[label=e2]{a.w.vandervaart@tudelft.nl}}
\address{TU Delft \\ \printead{e2}}

\thankstext{t1}{The research leading to these results 
is partly financed by a Spinoza prize awarded
 by the Netherlands Organisation for Scientific Research (NWO).}

\runauthor{S.E.M.P. Franssen and A.W. van der Vaart}

\end{aug}

\begin{abstract}
The Pitman-Yor process is a random probability distribution, that can be used 
as a prior distribution in a nonparametric Bayesian analysis.
The process is of species sampling type and generates discrete distributions, 
which yield of the order $n^\sigma$ different values (``species'') in
a random sample of size $n$, if the type $\sigma$ is positive.
Thus this type parameter can be set  to target true distributions of various levels of discreteness,
making the Pitman-Yor process an interesting prior in this case.
It was previously shown that the resulting posterior distribution is consistent if and only if the
true distribution of the data is discrete. In this paper we derive the distributional limit of the posterior distribution, in the form
of a (corrected) Bernstein-von Mises theorem, which previously was 
known only in the continuous, inconsistent case. It turns out that the Pitman-Yor
posterior distribution has good behaviour if the true distribution of the data 
is discrete with atoms that decrease not too slowly. Credible sets
derived from the posterior distribution provide valid frequentist confidence sets in this case. For a general
discrete distribution, the posterior distribution, although consistent, may contain a bias which does not converge to zero at the $\sqrt{n}$ rate
and invalidates posterior inference.
We propose a bias correction that solves this problem.  We also consider the effect of estimating the type parameter from 
the data, both by empirical Bayes and full Bayes methods. 
In a small simulation study we illustrate that without bias correction the coverage of credible sets can be arbitrarily low,
also for some discrete distributions.

\end{abstract}

\end{frontmatter}

\section{Introduction}
The Pitman-Yor process \cite{PitmanandYor(1997),PermanPitmanandYor(1992)}
is a random probability distribution, which can be used as a prior distribution in a nonparametric Bayesian analysis.
It is characterised by a \emph{type} parameter $\sigma$, which in this paper we take to be positive. 
The Pitman-Yor process of type $\sigma=0$
is the Dirichlet process \cite{Ferguson(1974)}, which is well understood, while negative types correspond to finitely discrete distributions
and were considered in \cite{DeBlasiLijoiandPrunster(2013)}.
The Pitman-Yor process is also known as the two-parameter Poisson-Dirichlet Process,
 is an example of a Poisson-Kingman process \cite{Pitman(2003)}, and a species sampling process
of Gibbs type~\cite{deBlasi2015}.

The easiest definition is through \emph{stick-breaking} (\cite{PermanPitmanandYor(1992),IshwaranJames}), as follows.
The family of nonnegative Pitman-Yor processes is given by three parameters: a number $\sigma \in [0,1)$, a number $M > -\sigma$
and an atomless probability distribution $G$ on some measurable space $(\X,\A)$.
We say that a random probability measure $P$ on $(\X,\A)$ is a Pitman-Yor process (of nonnegative type), 
denoted $P \sim \PY\left(\sigma, M, G \right)$, if $P$ can be represented as
\[
P = \sum_{i = 1}^\infty W_i \delta_{\theta_i},
\]
where $W_i = V_i \prod_{j = 1}^{i - 1} (1 - V_j)$ for $V_i \iid\Beta\left( 1 - \sigma, M + i \sigma\right)$, independent of 
$\theta_i \iid G$, and $\Beta$ the beta distribution.

It is clear from this definition that the realisations of $P$ are discrete probability measures, with countably many
atoms at random locations, with random weights. If one first draws $P\sim \PY\left(\sigma, M, G \right)$, and next given $P$ a random sample
$X_1,\ldots, X_n$ from $P$, then ties among the latter observations are possible, or even likely. It is known (\cite{Pitman(2003)})
that the number $K_n$ of different values among $X_1,\ldots, X_n$ is  almost surely of the order $n^\sigma$ if $\sigma>0$, whereas it
is logarithmic in $n$ if $\sigma=0$. This suggests that the Pitman-Yor process is a reasonable prior distribution
for a dataset in which similar patterns are expected (or observed). In particular, when a large number of clusters is expected,
a Pitman-Yor process of positive type could be preferred over the standard Dirichlet prior, which corresponds to $\sigma=0$.
Applications in genetics or topic modelling can be found in \cite{Wood2009, Teh2006, Goldwater2005, Arbel2018}. The Pitman-Yor 
process has also been proposed as a prior for estimating 
the probability that a next observation is a new species \cite{favaro2021nearoptimal}, with applications
in e.g.\ forensic statistics \cite{Cereda2017,Cereda2019}. The papers \cite{Camerlenghietal,Lijoietal2019}
study hierarchical versions of Pitman-Yor processes, which are useful  to discover structure in data beyond clustering.

In this paper we consider the properties of the Pitman-Yor posterior distribution to estimate
the distribution of a random sample of observations. By definition this posterior distribution is 
the conditional distribution of $P$ given $X_1,\ldots, X_n$ in the Bayesian hierarchical
model $P\sim \PY\left(\sigma, M, G \right)$ and  $X_1,\ldots, X_n\given P\iid P$. We assume that in reality the observations
$X_1,\ldots, X_n$ are an i.i.d. (i.e.\ independent and identically distributed) 
sample from a distribution $P_0$ and investigate the use of the posterior distribution 
for inference on $P_0$. It was shown in~\cite{James2008, deBlasi2015} that in this setting, as $n\rightarrow\infty$,
\begin{equation}
\label{EqInconsistency}
  P \given X_1, \ldots, X_n \rightsquigarrow \delta_{(1 - \lambda)P_0^d + \lambda(1 - \sigma) P_0^c + \sigma \lambda G},
\end{equation}
where $\rightsquigarrow$ denotes weak convergence of measures,  $\delta_Q$ denotes the Dirac measure at the probability distribution $Q$, and 
$P_0 = (1 - \lambda) P_0^d + \lambda P_0^c$ is the decomposition of $P_0$ in its discrete component $(1-\lambda) P_0^d$ and the remaining (atomless) part $\lambda P_0^c$.
In the case that $P_0$ is discrete, we have $\lambda=0$ and the measure 
$(1 - \lambda)P_0^d + \lambda(1 - \sigma) P_0^c + \sigma \lambda G$ in the right side reduces to $P_0^d=P_0$,  
and hence \eqref{EqInconsistency} expresses that the posterior distribution collapses asymptotically to the Dirac measure at $P_0$. The posterior distribution is said to be
consistent in this case. However, if $P_0$ is not discrete, then
the posterior distribution recovers $P_0$ asymptotically only if $\s=0$ (the case of the Dirichlet prior) or if $G = P_0^c$. The last case will typically
fail and hence in the case that $\s>0$ the posterior distribution will typically be consistent if and only if $P_0$ is discrete. This reveals the Pitman-Yor prior
of positive type as a reasonable prior only for discrete distributions.

Besides for recovery, a posterior distribution is used to express remaining uncertainty, for instance in the
form of a credible (or Bayesian confidence) set. To justify such a procedure from a non-Bayesian point of view,
the posterior consistency must be refined to a distributional result of Bernstein-von Mises type. Such a result was obtained
by \cite{James2008} in the case that the true distribution $P_0$ is atomless, 
the case that the posterior distribution is inconsistent and the Pitman-Yor
prior is better avoided. In the present paper we study the case of general distributions $P_0$, including the case of most interest
that $P_0$ is discrete. It turns out that discreteness per se is not enough for valid inference, but it is also needed that the weights of the
atoms in $P_0$ decrease fast enough. In the other case, ordinary Bayesian credible sets are not valid confidence sets. For
the latter case our result suggests a bias correction.

Since the type parameter $\sigma$ determines the number of distinct values in a sample from the prior, it might
be interpreted as influencing the discreteness of the prior, smaller $\sigma$ favouring fewer distinct values and hence 
a more discrete prior. In the asymptotic result the type parameter plays only a secondary role. 
At first thought counterintuitively,
a larger $\sigma$, which gives a less discrete prior, increases the bias in the posterior distribution that arises when the
atoms in $P_0$ decrease too slowly. 

In practice one may prefer to estimate the type parameter from the data. The empirical Bayes method
maximizes the marginal likelihood of $X_1,\ldots, X_n$ in the Bayesian setup over $\sigma$. We show 
that in the consistent case, substitution of
this estimator in the posterior distribution for given type parameter does not change the asymptotics 
of the Pitman-Yor posterior. Alternatively, we may equip $\sigma$ itself with a prior, resulting in  a
mixture of Pitman-Yor processes as a prior for $P$. We show that this too results in the same
posterior behaviour. Thus estimating the type parameter does not solve the inconsistency problem. 

We can conclude that the Pitman-Yor process is an appropriate prior for estimating a distribution 
only if the sizes of the atoms of this distribution decrease sufficiently rapidly. Our results show that
the speed of decay depends on the aspect of interest, for instance different for the distribution function
than for the mean.

Our results depend heavily on the characterisation of the posterior distribution given in \cite{Pitman1996b}
(see Section~\ref{SectionProofs}).

\section{Main result}
The nonparametric maximum likelihood estimator of the distribution $P$ of a sample of observations $X_1,\ldots, X_n$
is the empirical measure $\PP_n=n^{-1}\sum_{i=1}^n\delta_{X_i}$, the discrete uniform measure on the observations.
Therefore in analogy with the case of classical parametric models (e.g.\ Theorem~10.1 and page~144 in \cite{vanderVaart(1998)}),
in this setting a Bernstein-von Mises theorem would give the approximation of the posterior
distribution of $\sqrt n(P-\PP_n)$ given $X_1,\ldots, X_n$ by the normal distribution obtained as the limit of $\sqrt n (\PP_n-P_0)$.
To give a precise meaning to such a distributional statement, we may evaluate all the measures
involved on a collection of sets, and interpret $\sqrt n(P-\PP_n)$ and $\sqrt n (\PP_n-P_0)$ as stochastic processes
indexed by sets. For instance, in the case that the sample space is the real line, we could use the sets $(-\infty,t]$, for $t\in\RR$,
corresponding to the distribution functions of the measures $P$, $\PP_n$ and $P_0$.

More generally, we may evaluate these measures on measurable functions $f: \X\to\RR$, as
$$Pf=\int f\,dP,\qquad \PP_nf=\int f\,d\PP_n=\frac 1n \sum_{i=1}^nf(X_i),\qquad P_0f=\int f\,dP_0.$$
Given a collection $\F$ of such functions, the Bernstein-von Mises can then address the distributions of the 
stochastic processes $\bigl\{\sqrt n(Pf-\PP_nf): f\in\F\bigr\}$ and $\bigl\{\sqrt n (\PP_nf-P_0f): f\in\F\bigr\}$, the first
one conditionally given the observations $X_1,\ldots, X_n$.
For instance, in the case that $\X=\RR$, we might choose the collection $\F$
to consist of all indicator functions $x\mapsto 1_{x\le t}$, for $t$ ranging over $\RR$, but we can
also add the identify function $f(x)=x$, yielding the means of the measures.

For a  set $\F$ of finitely many functions, these processes are just vectors in Euclidean space 
and their distributions can be evaluated as usual. Furthermore, the  limit law of $\bigl\{\sqrt n (\PP_nf-P_0f): f\in\F\bigr\}$
is a multivariate normal distribution, in view of the multivariate central limit theorem (provided $P_0f^2<\infty$, for every $f\in\F$).
It is convenient to write the latter as the distribution of a Gaussian process
$\{\GG_{P_0}f: f\in\F\}$, determined by its mean  and covariance function
$$\EE \GG_{P_0}f=0\qquad \EE \GG_{P_0}f\GG_{P_0}g=P_0(f-P_0f)(g-P_0g).$$
The process $\GG_{P_0}$ is known as a {\sl $P_0$-Brownian bridge}
(see e.g.\ \cite{Pollard(1984), WCEP,vanderVaart(1998)}).

An appropriate generalisation (and strengthening) of the central limit theorem 
to sets $\F$ of infinitely many functions is Donsker's theorem (e.g.\ \cite{vanderVaart(1998)}, Chapter~19). The Bernstein-von Mises theorem can
be strengthened in a similar fashion. For the case of indicator functions on the real line, Donsker's theorem
was derived by \cite{Donsker}, and the corresponding Bernstein-von Mises theorem for the Dirichlet process by \cite{Lo(1983),Lo(1986)}.
A precise formulation (in the general case, which is not more involved than the real case) is as follows.

A class of functions $\F$ is called \emph{$P_0$-Donsker} if the sequence $\sqrt{n} \left( \PP_n - P_0 \right) $ converges in distribution
to a tight, Borel measurable element in the metric space  $\ell^\infty(\F)$ of bounded functions
$z: \F\to \RR$, equipped with the uniform norm $\|z\|_\F=\sup_{f\in\F}|z(f)|$. The limit is then a version of the Gaussian process $\GG_{P_0}$. 
The Bernstein-von Mises theorem involves conditional convergence in distribution given the observations
$X_1,\ldots, X_n$, which is best expressed using a metric. 
The bounded Lipschitz metric (see for example~\cite{WCEP}, Chapter 1.12) is convenient, and leads
to defining conditional convergence in distribution of the sequence $\sqrt n(P-\PP_n)$ in $\ell^\infty(\F)$ given $X_1,\ldots,X_n$ to $\GG_{P_0}$ as
\[
 \sup_{h \in \text{BL}_1} \Bigl| \EE \Bigl(h\bigl(\sqrt n(P-\PP_n)\bigr)\given X_1,\ldots, X_n\Bigr) - \EE h(\GG_{P_0}) \Bigr|\rightarrow 0,
\]
where the convergence refers to the i.i.d.\ sample $X_1,X_2,\ldots$ from $P_0$, 
and can be in (outer) probability or almost surely. The supremum is taken over the set  $\text{BL}_1$ of all functions $h: \ell^\infty(\F) \rightarrow [0,1]$ 
such that  $|h(z_1) - h(z_2)| \leq \| z_1 - z_2 \|_{\F}$, for all $z_1, z_2\in\linfty(\F)$. For simplicity of notation and easy interpretation, we  write the preceding
display as 
$$\sqrt n(P-\PP_n)\given X_1,\ldots, X_n\weak \GG_{P_0}.$$
Conditional convergence in distribution of other
processes is defined and denoted similarly. For finite sets $\F$, the complicated definition using the bounded Lipschitz metric
reduces to ordinary weak convergence of random vectors. Also, a finite set $\F$ is $P_0$-Donsker if and only if $P_0f^2<\infty$, for
every $f\in\F$. There are many examples of infinite Donsker classes (see e.g.\ \cite{WCEP}), with the set of indicators of cells $(-\infty,t]$ as the
classical example.

We are ready to formulate the main result of the paper.
Let $\tilde X_1,\tilde X_2,\ldots$ be the distinct values in $X_1,X_2,\ldots$ in the order of appearance, let
$K_n$ be the number of distinct elements among $X_1,\ldots, X_n$, and set 
\begin{equation}
\label{EqDefinitionPTilde}
\tilde\PP_n=\frac{1}{K_n}    \sum_{i = 1}^{K_n} \delta_{\tilde{X}_i}.
\end{equation}
All limit results refer to a sample $X_1,X_2,\ldots, X_n$ drawn from a measure $P_0$.
This can always be written 
as $P_0 = (1 - \lambda) P_0^d + \lambda P_0^c$, where $P_0^d$ is a discrete and $P_0^c$ an atomless distribution
and $\lambda \in [0,1]$ is the weight of the discrete part in $P_0$. The decomposition is unique unless $\lambda=0$ or $\lambda=1$, when
$P_0^c$ or $P_0^d$ is arbitrary.

\begin{theorem}
\label{thm:PYBVM}
Let $P_0 = (1 - \lambda) P_0^d + \lambda P_0^c$ where $P_0^d$ is a discrete  and $P_0^c$ an atomless probability distribution. 
The posterior distribution of $P$ in the model  $P\sim \PY \left( \sigma, M, G \right)$ and $X_1, \ldots, X_n \given P \iid P$ satisfies
for every finite collection $\F$ of functions with $(P_0+G)f^2<\infty$, for every $f\in\F$, almost surely under  $P_0^\infty$,
\begin{align*} \sqrt{n} \Bigl(&     P    -    \PP_n    -    \frac{\sigma K_n}{n} (G    -   \tilde\PP_n )\Bigr) \Big| \  X_1, \ldots, X_n \\
&\qquad\rightsquigarrow 
\sqrt{1 - \lambda}\, \GG_{P_0^d} + \sqrt{(1 - \sigma) \lambda}\,\GG_{P_0^c} 
+ \sqrt{\sigma(1 - \sigma)\lambda} \,\GG_G\\
&\qquad\qquad + \sqrt{(1 - \sigma \lambda)\sigma\lambda} 
\Bigl(    \frac{   (1 - \lambda) P_0^d + (1 - \sigma) \lambda P_0^c    }{ 1 - \sigma \lambda }    -     G\Bigr)\,Z_1\\
&\qquad\qquad\qquad+ \frac{\sqrt{(1 - \sigma) \lambda(1 - \lambda)}}{\sqrt{1 - \sigma \lambda}} (P_0^c - P_0^d) \,Z_2.
\end{align*}
Here $\GG_{P_0^d}$, $\GG_{P_0^c}$ and $\GG_G$ are independent Brownian bridge processes, independent
of the independent standard normal variables  $Z_1$ and $Z_2$.
More generally this is true, with convergence in $\linfty (\F)$ in probability,  for every $P_0$-Donsker class of functions $\F$
for which the $\PY\left( \sigma, \sigma, G \right)$ process satisfies the central limit theorem in $\ell^\infty\left( \mathcal{F} \right)$.
If in addition $P_0^* \| f - P_0 f \|_{\F}^2 < \infty$, then the convergence is also $P_0^\infty$-almost surely.
\end{theorem}

The proof of the theorem is deferred to Section~\ref{SectionProofMain}. The condition that the 
$\PY\left( \sigma, \sigma, G \right)$ process satisfies the central limit theorem in $\ell^\infty\left( \mathcal{F} \right)$, is satisfied,
for instance, for all classes $\F$ that are suitably measurable with finite uniform entropy integral and for all classes $\F$
with finite $G$-bracketing integral. This follows from Theorems~2.11.9 or 2.11.1 in \cite{WCEP}.

The limit process in the theorem is Gaussian, but it is the $P_0$-Brownian bridge $\GG_{P_0}$ only if $\lambda=0$, i.e.\
if $P_0=P_0^d$ is discrete. In addition, the behaviour of the Pitman-Yor posterior deviates from  the ``desired''
behaviour by the presence on the left side of the term 
\begin{equation}
\label{EqBias}
\sqrt n\, B_n(f):=\frac{\sigma K_n}{\sqrt n}    (G    -      \tilde\PP_n).
\end{equation}
Given the observations $X_1,\ldots, X_n$ this term is deterministic, and we can only expect it to disappear if 
$K_n/\sqrt n$ tends to zero. While $K_n/n\ra 0$ almost surely for any discrete distribution $P_0$, 
the more stringent convergence to zero of
$K_n/\sqrt n$ is valid only if the sizes of the atoms of $P_0$ decrease fast enough. This relationship was made
precise in \cite{Karlin1967} (also see the corollary below) in terms of the function
\begin{equation}
\a_0(u)=\#\{ x: 1/P_0\{x\}\le u\}.
\label{EqDefinitionAlpha}
\end{equation}
If $\alpha_0$ is regularly varying at $u=\infty$ (in the sense of Karamata, see e.g.\ \cite{Binghametal} or the appendix to \cite{deHaan})
with exponent $\g_0\in(0,1)$, then $K_n/\a_0(n)\ra \Gamma(1-\g_0)$, almost surely, and $\a_0(n)$ is $n^{\g_0}$ up to a slowly varying
factor. In this case, for $K_n/\sqrt n$ to tend to zero, it is necessary that the exponent be smaller  than 1/2 and sufficient that  it is strictly smaller  than 1/2.
For instance, if the ordered atoms $P_0\{x_j\}$ of $P_0$  decrease proportionally to $1/j^{\a}$, then  $K_n/\sqrt n\ra 0$ 
in probability if and only if $\a>2$.

For bounded functions $f$, the convergence $K_n/\sqrt n\ra 0$ is also enough to drive the additional term \eqref{EqBias}
to zero, as the terms $(G-\tilde \PP_n)f$ will remain bounded in that case.
For unbounded functions $f$, a still more stringent condition on $P_0$ is needed to make the term $(K_n/\sqrt n)\tilde \PP_nf$ go away.
For instance, for the posterior mean of a distribution on $\NN$ with atoms $P_0\{j\}$ of the order $1/j^{\a}$,
the next corollary implies that $\a>4$ is needed.

We conclude that  for a  large class of discrete distributions $P_0$, but not all, the Bernstein-von Mises theorem takes 
its standard form, and this also depends on which aspect of the posterior distribution we are interested in.

\begin{corollary}
\label{CorOne}
Under the conditions of Theorem~\ref{thm:PYBVM}, if $P_0$ is a discrete probability distribution, then 
$\sqrt{n} \bigl( P -  \PP_n  - (\sigma K_n/n)(G - \tilde\PP_n)\bigr) | X_1, \ldots, X_n \rightsquigarrow \GG_{P_0}$
in $\ell^\infty ( \F)$, in probability or almost surely. 
\begin{itemize}
\item[(i)] If the class of functions $\F$ is uniformly bounded and 
the atoms $\{x_j\}$  of $P_0$ satisfy $P_0\{x_j\}\le C j^{-\a}$, for some constants $C$ and $\a>2$,
then also
$\sqrt{n} ( P -  \PP_n) | X_1, \ldots, X_n \rightsquigarrow \GG_{P_0}$, in probability.
If the class of functions $\F$ is uniformly bounded and 
the function $u\mapsto \a_0(u)=\#\{ x: 1/P_0\{x\}\le u\}$  is regularly varying at $u=\infty$ of exponent strictly smaller than 1/2, 
then this is also true almost surely.
\item[(ii)] If the atoms $\{x_j\}$ of $P_0$ and the function $f$ satisfy $P_0\{x_j\}\le C j^{-\a}$
and $f(x_j)\asymp j^p$, for some $p>0$,  then $\sqrt{n} ( P -  \PP_n)f | X_1, \ldots, X_n \rightsquigarrow \GG_{P_0}f$, in probability
if $\a>2p+2$.
\end{itemize}
\end{corollary}

\begin{proof}
The first assertion merely specializes the limit in Theorem~\ref{thm:PYBVM} to the case
of a discrete distribution, by setting $\lambda=0$. Assertions (i) and (ii) follow from this
if the term \eqref{EqBias} tends to zero, in probability or almost surely.

For bounded functions $f$, as assumed in (i), the term \eqref{EqBias} tends to zero provided $K_n/\sqrt n$
tends to zero. The almost sure convergence is immediate from \cite{Karlin1967}, Theorems~9 and~1`, which show that
$K_n/\a_0(n)\ra \Gamma(1-\g_0)$, almost surely, for $\g_0$ the exponent of regular variation.
For the convergence in probability, we note that $K_n=\sum_{j=1}^\infty 1_{j\in \{X_1,\ldots, X_n\}}$, whence
$\EE K_n=\sum_{j=1}^\infty \bigl(1-(1-P_0\{x_j\})^n\bigr)$. By the inequality $(1-p)^n\ge 1-np$, for $p\ge 0$,
we find that $\EE K_n\le \sum_{j=1}^\infty(nCj^{-\a}\wedge 1)$, which can be seen to be $o(\sqrt n)$ if $\a>2$.

The assertion in (ii)  follows provided $K_n/\sqrt n\ra 0$ and $(K_n/\sqrt n)\tilde\PP_nf\ra 0$, in probability.
Reasoning as before, we find
$$\EE \Bigl(\frac{K_n}{\sqrt n}\tilde\PP_nf\Bigr)=\frac 1{\sqrt n}\sum_{j=1}^\infty f(x_j)\bigl(1-(1-P_0\{x_j\})^n\bigr).$$
Under the given assumptions on $f$ and the atoms, this is bounded above by
$$\frac 1{\sqrt n}\int_{C^{1/a}}^\infty u^p\bigl(1-\Bigl(1-\frac C{u^\a}\Bigr)^n\Bigr)\,du
=\frac{n^{(p+1)/\a}}{\sqrt n}\int_{C/n}^\infty v^{(p+1)/\a-1}\Bigl(1-\Bigl(1-\frac C{nv}\Bigr)^n\Bigr)\,dv.$$
The integrand is bounded above by $Cv^{(p+1)/\a-1}$ and hence the integral converges near 0.
By again the inequality $(1-p)^n\ge 1-np$, for $p\ge 0$,  the integrand is also bounded
above by $v^{(p+1)/\a-2}$ and hence the integral converges near infinity if $(p+1)/\a<1$.
The middle part of the integral always gives a non-vanishing contribution and hence the full
expression can tend to zero only if the leading factor tends to zero. This is true
under the more stringent condition that  $(p+1)/\a<1/2$.
\end{proof}

For $\lambda=1$ and $P_0=P_0^c$, Theorem~\ref{thm:PYBVM}  was obtained by \cite{James2008}. 
In this case all observations are distinct and the
left side of the theorem reduces to $\sqrt n \bigl(P-(1-\sigma)\PP_n-\sigma G\bigr)$, since $K_n=n$. 
As noted in the introduction, the posterior distribution is not even consistent, i.e.\ 
the asymptotic limit is ``wrong'' even without the $\sqrt n$ multiplier.

The Bernstein-von Mises theorem is important for the validity of credible sets. A credible interval for 
$Pf$, for a given function $f$, could be formed as the interval between two quantiles of the marginal posterior distribution
of $Pf$ given $X_1,\ldots, X_n$.  For instance, for $f$ equal to the indicator of a given set $A\in\A$, this gives a credible interval
for a probability $P(A)$, and for $f(x)=x$, we obtain a credible interval for the mean. Simultaneous credible sets,
for instance a credible band for a distribution function can be obtained similarly.

By the inconsistency of the posterior distribution in the case that the
true distribution possesses a continuous component ($\lambda >0$), there is no hope that in this case such an
interval for $Pf$ will cover a true value $P_0f$ with the desired probability. However, also in the
case of a discrete distribution $P_0$, the coverage may not tend to the nominal value,
due to the presence of the bias term \eqref{EqBias}. We need at least that $K_n/\sqrt n$ tends to
zero, and more for unbounded functions $f$.

Because the bias $B_n(f)=(\s K_n/n)(Gf-\tilde \PP_nf)$ is observed (and $\s$ and the center measure $G$ are fixed by our
prior choices), it is possible to correct a credible interval by shifting it by minus this amount.
Thus for $Q_{n,\a}(f)$ the $\a$-quantile of the posterior distribution of $Pf$ given $X_1,\ldots, X_n$,
we consider both the credible intervals
$\bigl[Q_{n,\a}(f), Q_{n,\b}(f)\bigr]$ and corrected intervals $\bigl[Q_{n,\a}(f)-B_n(f), Q_{n,\b}(f)-B_n(f)\bigr]$,
for given $\a<\b$.

\begin{corollary}
\label{CorCredible}
Under the conditions of Theorem~\ref{thm:PYBVM}, if $P_0$ is a discrete probability distribution, then 
$\Pr_{P_0}\bigl(Q_{n,\a}(f)-B_n(f)\le P_0f\le  Q_{n,\b}(f)-B_n(f)\bigr)\rightarrow \b-\a$, for every
$f$ with $(P_0+G)f^2<\infty$. If $\sqrt n B_n(f)\ra 0$, in probability, then also
$\Pr_{P_0}\bigl(Q_{n,\a}(f)\le P_0f\le  Q_{n,\b}(f)\bigr)\rightarrow \b-\a$, for every such $f$.
For bounded functions $f$, the latter is true if the atoms of $P_0$ satisfy
$P_0\{x_j\}\le C j^{-\a}$, for some constants $C$ and $\a>2$. For $f(x)=x$, this is true for $\a>4$.
\end{corollary}

\begin{proof}
The $\a$-quantile $Q_{n,\a}(f)$ of the posterior distribution of $Pf$ is equal to $n^{-1/2}\bar Q_{n,\a}(f)+\PP_nf+B_n(f)$, for
$\bar Q_{n,\a}(f)$ the $\a$-quantile of the posterior distribution of $\sqrt n\bigl(Pf-\PP_nf-B_n(f)\bigr)$. 
By Theorem~\ref{thm:PYBVM}, the latter posterior distribution tends
to a normal distribution with mean zero and variance $\tau^2(f)=\var \GG_{P_0}f$. It follows that
$$Q_{n,\a}(f)=\PP_nf+B_n(f)+\frac{\tau(f)}{\sqrt n}\xi_\a+o_P\Bigl(\frac1{\sqrt n}\Bigr),$$
where $\xi_\a$ is the $\a$-quantile of the standard normal distribution.
Thus the event $P_0f\ge  Q_{n,\a}(f)-B_n(f)$ can be rewritten as
$P_0f\ge\PP_nf+\tau(f)/\sqrt n \,\xi_\a+o_P(n^{-1/2}$. The probability of the latter event tends to
tends to $1-\a$, by the central limit theorem applied to $\sqrt n (\PP_nf-P_0f)$.

If $\sqrt n B_n(f)$ tends to zero in probability, then in the preceding display $B_n(f)$ can be incorporated
into the $o_P(n^{-1/2})$ remainder term, and the remaining argument works for the uncorrected interval as well.

The final assertions follow from Corollary~\ref{CorOne}.
\end{proof}


\begin{example}\label{Example_of_noncoverage}
The following explicit counterexample illustrates that the coverage can fail.
Let $G $ be the normal distribution with both mean and variance $1$, let $P_0 = \sum_{j = 1}^\infty p_j \delta_j$, for $p_j = {6}/{(\pi j)^2}$, 
and consider the function $f = \II_{(1,\infty)}-\II_{(-\infty,1]}$. Since $Gf = 0$, we get 
$ (\sigma K_n/\sqrt n) (G - \tilde\PP_n)f = (\sigma/\sqrt{n}) \sum_{i = 1}^{K_n} f(\tilde{X}_i)$.
Eventually the atom $\{1\}$ will be among the observations. Since $f(1)=-1$ and $f(j)=1$ for all atoms $j\ge 2$,
$(\sigma K_n/\sqrt n) (G - \tilde\PP_n)f= (\sigma/\sqrt{n}) ( -1 + (K_n - 1))  \rightarrow \sigma \sqrt{6/\pi}$, almost surely,
by~\cite{Karlin1967}, Theorem~8, Theorem~1' and Example~4.  
The coverage of the uncorrected interval $\bigl[Q_{n,\a}(f), Q_{n,\b}(f)\bigr]$ will tend to $\Phi(-\xi_\a-\sigma \sqrt{6/\pi})
-\Phi(-\xi_\b-\sigma \sqrt{6/\pi})$.
\end{example}

The joint convergence in collections of functions $f$ allows to study simultaneous credible sets and credible bands, besides univariate intervals.
For instance, in the case that the sample space is the real line, we can take $\F$ equal to the set of all indicators of cells $(-\infty,t]$,
and obtain a credible band for the distribution function $F_0(t)=P_0(-\infty,t]$, as follows. 
Let $F(t)=P(-\infty,t]$ be the  distribution function of the posterior process, and for $m_n(t)$ and $s_n(t)$ two functions
dependent on $X_1,\ldots, X_n$, let $\xi_{n,\a}$ be the $\a$-quantile
of the posterior distribution of $\sup_{t\in\RR}\bigl|(F(t)-m_n(t))/s_n(t)\bigr|$. Consider the credible band of functions
$$C_n(\a):=\bigl\{F:  m_n(t)-\xi_{n,1-\a} s_n(t)\le F(t)\le m_n(t)+\xi_{n,1-\a} s_n(t), \forall t\bigr\}.$$
Possible choices for the functions $m_n$ and $s_n$ are the pointwise posterior mean $m_n(t)=\EE\bigl(F(t)\given X_1,\ldots, X_n\bigr)$
and the pointwise posterior standard deviation $s_n(t)=\sd\bigl(F(t)\given X_1,\ldots, X_n\bigr)$. The quantiles $\xi_{n,\a}$
will typically be computed approximately from an MCMC sample from the posterior distribution, or approximated using
tables for the limiting Brownian bridge process.

\begin{corollary}
If $P_0$ is a discrete probability distribution with  atoms such that
$P_0\{x_j\}\le C j^{-2-\e}$, for some constants $C$ and $\e>0$, then 
$\Pr_{P_0}\bigl(F_0\in C_n(\a)\bigr)\rightarrow 1-2\a$.
\end{corollary}

\begin{proof}
Because the class of indicator functions is Donsker,  both the classical empirical process
$\{\sqrt n(\FF_n-F_0)(t): t\in\RR\}$ and the posterior empirical process  $\{\sqrt n (F-\FF_n)(t): t\in\RR\}\given X_1,\ldots,X_n$
tend to the process $\GG_U\circ F_0$, for $\GG_U$ a standard (classical)  Brownian bridge process, 
by Theorem~\ref{thm:PYBVM}. The result follows from this along the same lines
as the proof of Corollary~\ref{CorCredible}.
\end{proof}

The bias term \eqref{EqBias} vanishes as $\sigma\downarrow 0$, which is in agreement with the fact that in this case the
Pitman-Yor prior approaches the Dirichlet prior, which is well known to give asymptotically correct inference for
any distribution $P_0$. The bias term increases with $\s$, which is counterintuitive, as the bias appears only for
heavy-tailed $P_0$ (having many large atoms), while large $\sigma$ gives more different atoms in the prior. 

One might hope that a data-dependent choice of $\sigma$ could solve this bias problem. The empirical Bayes
method is to estimate $\s$ by the maximum likelihood estimator based on observing $X_1,\ldots,X_n$ in the
Bayesian model, i.e.\ the maximiser of the marginal likelihood, and plug this into the posterior distribution
of $P$ for known $\s$. The hierarchical Bayes method is to put
a prior on $\s$, and given $\s$, put the Pitman-Yor prior on $P$. 
Disappointingly, these methods do not change the limit behaviour of the posterior distribution of $P$. 
This is explained by the fact that these methods yield a reasonable estimator of a value of $\sigma$
connected to the discreteness of the true distribution $P_0$, and we already noted the
counterintuitive fact that a better match of discreteness does not solve the bias problem, but
even makes it worse.

A sample $X_1,\ldots, X_n$ from a realisation of the Pitman-Yor process induces a (random) partition
of the set $\{1, 2,\ldots, n\}$ through the equivalence relation $i\equiv j$ if and only if $X_i=X_j$.
An alternative way to generate the sample is to generate first the partition and next attach
to each set in the partition a value generated independently from the center measure $G$
(see e.g.\ \cite{FNBI}, Lemma~14.11 for a precise statement), duplicating this as many times as there are indices in the set, in order
to form the observations $X_1,\ldots, X_n$.
Because the parameter $\s$ enters only in creating the
partition, the partition is a sufficient statistic for $\s$. Because of exchangeability, the vector
$(N_{n,1},\ldots, N_{n,K_n})$ of cardinalities of the partitioning sets is already sufficient for $\s$ and hence
the empirical Bayes estimator and posterior distribution of $\s$ based on observations $(X_1,\ldots,X_n)$
or on  observations $(K_n,N_{n,1},\ldots, N_{n,K_n})$ are the same. 

The likelihood function for $\s$ is therefore equal to the probability of a particular partition,
called the exchangeable partition probability function (EPPF). For the Pitman-Yor process this is
given by (see \cite{Pitman1996b}, or \cite[page 465]{FNBI})
\begin{equation}
\label{EqLikelihoodSigma}
p_\s(N_{n,1},\ldots, N_{n,K_n})=\frac{\prod_{i=1}^{K_n-1}(M+i\s)}{(M+1)^{[n-1]}}\prod_{j=1}^{K_n}(1-\s)^{[N_{n,j}-1]}.
\end{equation}
Here $a^{[n]}=a(a+1)\cdots(a+n-1)$ is the ascending factorial, with $a^{[0]}=1$ by convention. For the case that $M=0$,
it is shown in \cite{favaro2021nearoptimal},
that provided the partition is nontrivial ($1<K_n<n$), the maximiser $\hat\s_n$ of this likelihood exists.
Moreover, if the true distribution $P_0$ is discrete, with atoms satisfying, for $\a_0(u)=\#\{ x: 1/P_0\{x\}\le u\}$ 
and some $\s_0\in(0,1)$,
\begin{equation*}
\sup_{u>1} \frac{|\alpha_0(u)-Lu^{\s_0}|}{\sqrt{u^{\s_0}\log(eu)}}<\infty,
\end{equation*}
then \cite{favaro2021nearoptimal} shows that the maximum likelihood estimator satisfies $\hat\s_n={\s_0}+O_P(n^{-{\s_0}/2}\sqrt{\log n})$. Thus the
coefficient of regular variation ${\s_0}$ may be viewed as a true value of $\s$, identified by the maximum
likelihood estimator. 

For the following theorem we need only the consistency of $\hat\s_n$, which we prove in Section~\ref{SectionEstimatingType}
for general $M$, under the condition that $\a_0$ is regularly varying. 
We also consider the full Bayes approach, and show that 
the posterior distribution of $\s$ concentrates asymptotically around the empirical likelihood estimator,
and hence contracts to $\s_0$, under the same condition.

\begin{theorem}
\label{thm:PYBVMEstimatedSigma}
Let $P_0 = (1 - \lambda) P_0^d + \lambda P_0^c$ where $P_0^d$ is a discrete  and $P_0^c$ an atomless probability distribution. 
If $\hat\s_n$ are estimators based on $X_1,\ldots, X_n$ such that $\hat\s_n\ra\s_0$ in probability, and $P$ is the
posterior Pitman-Yor process of Theorem~\ref{thm:PYBVM},  then the process
\begin{align*} \sqrt{n} \Bigl(&     P    -    \PP_n    -    \frac{\hat\sigma_n K_n}{n} (G    -   \tilde\PP_n )\Bigr) \Big| \  X_1, \ldots, X_n 
\end{align*}
tends to the same limit process as in Theorem~\ref{thm:PYBVM} with $\s$ replaced by $\s_0$, in probability. If $P_0$ is discrete 
with atoms such that $\a_0$ given in \eqref{EqDefinitionAlpha} is regularly varying of exponent $\s_0\in(0,1)$,
then this is true for the maximum likelihood estimator $\hat \s_n$.  
Furthermore, in this case for $\Pi_\s$ a prior distribution on $\s$ with continuous positive density on $[0,1]$, the posterior distribution of $P$ in the model 
$\s\sim \Pi_\s$, $P\given \s\sim\PY \left( \sigma, M, G \right)$ and $X_1, \ldots, X_n \given P,\s \sim P$ satisfies the assertion
of Theorem~\ref{thm:PYBVM}, with $\s$ in the left side also interpreted as a random posterior variable
and $\s$ in the right side replaced by $\s_0$. Finally if $P_0$ possesses a nontrivial atomless component (i.e.\ $\lambda>0$),
then $\hat\s_n\ra \s_0:=1$.
\end{theorem}

The proof of the theorem is deferred to Section~\ref{Sectionthm:PYBVMEstimatedSigma}.
The final assertion of the theorem underlines again the deficiency of the Pitman-Yor process for
distributions with a continuous component, which is not solved by estimating the type parameter .
The type estimate tends to type 1 instead of the desired type 0 corresponding to the Dirichlet prior.

Besides the type parameter, the prior precision parameter $M$ could be replaced by a data-dependent
version. However, unlike the type parameter, this prior precision does not appear in the asymptotics of the
posterior distribution of $\sqrt n(P-\PP_n)$. Moreover, inspection of the proof of Theorems~\ref{thm:PYBVM}
and~\ref{thm:PYBVMEstimatedSigma} shows that the convergence in these theorems is uniform in
$M\ll \sqrt{n}$. Thus data-dependent $M$ will not lead to new insights.

In the case of a discrete distribution $P_0$ for which the atoms
decrease too slowly to ensure that $K_n/\sqrt n$ tends to zero, the bias term \eqref{EqBias}
could still tend to zero if $\tilde \PP_n\ra G$. However, we show below that $\tilde\PP_n (A)\ra 0$, for any set $A$ 
that contains only finitely many atoms of $P_0$, and hence such convergence is false in any reasonable sense.
Furthermore, the (in)consistency result \eqref{EqInconsistency} shows that in the case that $P_0$ possesses a continuous
component,  the center measure $G=P_0^c$ is the only choice for which the posterior distribution is even consistent.
A data-dependent center measure might achieve this, but in the present context would come
down to the original problem of estimating $P_0$. Hierarchical choices (and hence random) of the center measure are 
considered in \cite{Camerlenghietal,Lijoietal2019}, but with the different aim of finding hierarchical structures in the data.

\begin{lemma}
If $P_0$ is discrete with infinitely many support points,
then $\tilde \PP_nf\ra 0$ in probability for any bounded function $f$ with finite support.
Furthermore,   $\tilde \PP_nf\ra f_\infty$ in probability, for any bounded function $f$ for which there exists a number
$f_\infty$ such that $\sup_{x: P_0\{x\}<\delta}|f(x)-f_\infty|\ra0$, as $\delta\downarrow 0$.
\end{lemma}

\begin{proof}
Let $x_j$ be the atoms of $P_0$, ordered by decreasing size $p_j:=P_0\{x_j\}$ and set $f_j=f(x_j)$.
Arguing as in the proof of Corollary~\ref{CorOne}, we can obtain (for the variance also see \cite{Karlin1967}, formulas (39)--(40))
\begin{align*}
\EE (K_n\tilde\PP_nf)& =\sum_{j=1}^\infty f_j\bigl(1-(1-p_j)^n\bigr),\\
\var (K_n\tilde\PP_nf)& =\sum_{j=1}^\infty f_j^2\bigl[(1-p_j)^n-(1-p_j)^{2n}\bigr]\\
&\qquad+\mathop{\sum\sum}_{i\not=j} f_if_j\bigl[(1-p_i-p_j)^n-(1-p_i)^n(1-p_j)^n\bigr].
\end{align*}
For $f=1$ these expressions reduce to $\EE K_n$ and $\var K_n$. As all terms of the series in $\EE K_n$ tend to 1,
it can be seen that $\EE K_n\ra\infty$. Furthermore, it can be seen that $\var K_n\le \EE K_n$, as the terms in the 
second series in $\var K_n$ are negative and the terms of the first series are bounded above by the terms in $\EE K_n$.
Because the terms of the series
tend to zero as $n\ra\infty$, for fixed $i,j$,  and $f_j\ra f_\infty$, as $j\ra\infty$, for general
$f$ as in the second assertion of the lemma, the expressions are asymptotically equivalent to
$f_\infty\EE K_n+o(1)$ and $f_\infty^2\var K_n+o(1)$, as $n\ra\infty$. It follows that
$$\var\Bigl(\frac{K_n\tilde\PP_nf}{\EE (K_n\tilde \PP_nf)}-1\Bigr)=\frac{\var (K_n\tilde\PP_nf)}{\bigl(\EE( K_n\tilde PP_nf)\bigr)^2}
=\frac{f_\infty^2\var K_n+o(1)}{\bigl(f_\infty\EE K_n+o(1)\bigr)^2}.$$
Since $\var K_n\le \EE K_n\ra\infty$, the right side tends to zero if $f_\infty\not=0$. 
Then $K_n\tilde \PP_nf/\EE (K_n\tilde\PP_nf)\ra 1$, in probability. Taking $f=1$, we see that 
$K_n/\EE K_n\ra 1$, in probability. Combining the preceding, we conclude that
$\tilde\PP_nf/f_\infty\ra 1$, in probability.

If $f_\infty=0$, then it follows that $\EE (K_n\tilde\PP_nf)\ra 0$. Combination with the fact that $K_n\ra\infty$,
almost surely, gives that $\tilde\PP_nf\ra0$, in probability.

If $f$ has finite support, then the condition on $f$ in the second part holds with $f_\infty=0$ and hence
$\tilde\PP_nf\ra0$, in probability.
\end{proof}

\section{Numerical illustration}
To illustrate that credible sets can be off, we carried out  three simulation experiments, involving
three discrete true probability distributions $P_1, P_2, P_3$ on $\NN$. We focused on a credible
interval for the probability of the set $[2,\infty)$. The measure $P_1$ is finitely discrete
and given in  Table~\ref{def_p1}, while $P_2$ and $P_3$ are given by the formulas
\[P_2\{k\} \propto \frac{1}{k^2}, \qquad\qquad P_3 \{k\}\propto \frac{1}{k^{1.5}}.\]
By the results of \cite{Karlin1967}, as $n\ra\infty$ the number $K_n$ of distinct observations in a sample of size $n$ from
these distributions are asymptotically equal to 6, and proportional to $\sqrt{n}$ and  to $n^{2/3}$, respectively,  for $P_1$, $P_2$ and $P_3$.
Thus $K_n/\sqrt n$ tends to 0, a positive constant and $\infty$, respectively, and a bias
is expected for $P_2$ and $P_3$, but not for $P_1$, where $P_2$ is a boundary case.

\begin{table}[ht]
\caption{Probability distribution $P_1$}
\centering
\begin{tabular}{l | c c c c c c}\label{def_p1}
k         & 1 & 2& 3 & 4 & 5 & 6\\
$P_1(X= k)$ & 0.1 & 0.1 & 0.2 & 0.2 & 0.3 & 0.1
\end{tabular}
\end{table}

As prior parameters we used $\sigma=1/2$ and $M=1$ and $G$ the normal distribution with mean and variance 1.
The choice $M=1$ means that the prior is not biased exceedingly against the true distribution.

The Pitman-Yor posterior distribution can be simulated using the explicit representation given by
\cite{Pitman(1996b)} (see Section~\ref{SectionProofs}). Following Algorithm~1 from~\cite{Arbel2018}, we
truncated the infinite series in the representation at a finite value, ensuring that the total weight of the
tail is smaller than $n^{-1/2}$ so that the approximation is accurate within our context.
We simulated 10000 samples from each of $P_1$, $P_2$ and $P_3$ and for five different sample sizes: $n=10, 10^2, 10^3, 10^4, 10^5$. 
For each sample we computed a $95\%$ credible interval for $P[2,\infty)$ from
its marginal posterior distribution, constructed using the 0.025 and the 0.975 posterior quantiles.
We next computed coverage as the proportion of the 10000 replications that the true value, $P_1[2,\infty)$, $P_2[2,\infty)$ or $P_3[2,\infty)$,
belonged to the interval. We did the same with the credible interval shifted by 
$(\sigma K_n/\sqrt n)  \bigl(G[2,\infty)    -      \tilde\PP_n[2,\infty)\bigr)$, derived from \eqref{EqBias}.

Tables~\ref{table_1} and~\ref{table_2} summarise the results. For $P_1$ both
the corrected uncorrected intervals perform satisfactorily, whereas for $P_2$ and $P_3$ the uncorrected
intervals undercover, severely so for $P_3$, while the corrected intervals perform reasonably well, although not perfectly. 
The simulation results thus confirm the theoretical findings.

\begin{table}[ht]
\caption{Coverage of uncorrected posterior $95\%$ credible intervals}
\centering
\begin{tabular}{l | c c c c c}\label{table_1}
n  & 10 & 100 & 1000 & 10000 & 100000 \\
$P_1$ & 0.660 & 0.940 & 0.957 & 0.940 & 0.947\\
$P_2$ & 0.707 & 0.772 & 0.790 & 0.845 & 0.838\\
$P_3$ & 0.559 & 0.231 & 0.035 & 0.0 & 0.0
\end{tabular}
\end{table}

\begin{table}[ht]
\caption{Coverage of corrected posterior $95\%$ credible intervals}
\centering
\begin{tabular}{l | c c c c c} \label{table_2}
n  & 10 & 100 & 1000 & 10000 & 100000 \\
$P_1$ & 0.990 & 0.967 & 0.958 & 0.942 & 0.945\\
$P_2$ & 0.814 & 0.941 & 0.958 & 0.971 & 0.971\\
$P_3$ & 0.884 & 0.956 & 0.959 & 0.985 & 0.949
\end{tabular}
\end{table}

To illustrate the asymptotic normality of the posterior distribution, Figure~\ref{fig:densities}
shows density plots of the marginal posterior distribution of $P[2,\infty)$, given samples of various sizes from $P_1$.
The plots were computed from the 100000 replicates, using the {\tt R} ``density'' function. 
The normal approximation is satisfactory for $n=1000$, but the posterior is visibly skewed for $n=100$.

\begin{figure}
\caption{Density of the marginal posterior distribution of $P[2,\infty)$ based on $n$ observations from $P_1$,
for $n=10, 100, 1000$ (top row) and $n=10^4, 10^5$. The true value of the parameter is $P_1[2,\infty)=0.9$.}
 \includegraphics[width =3.8cm]{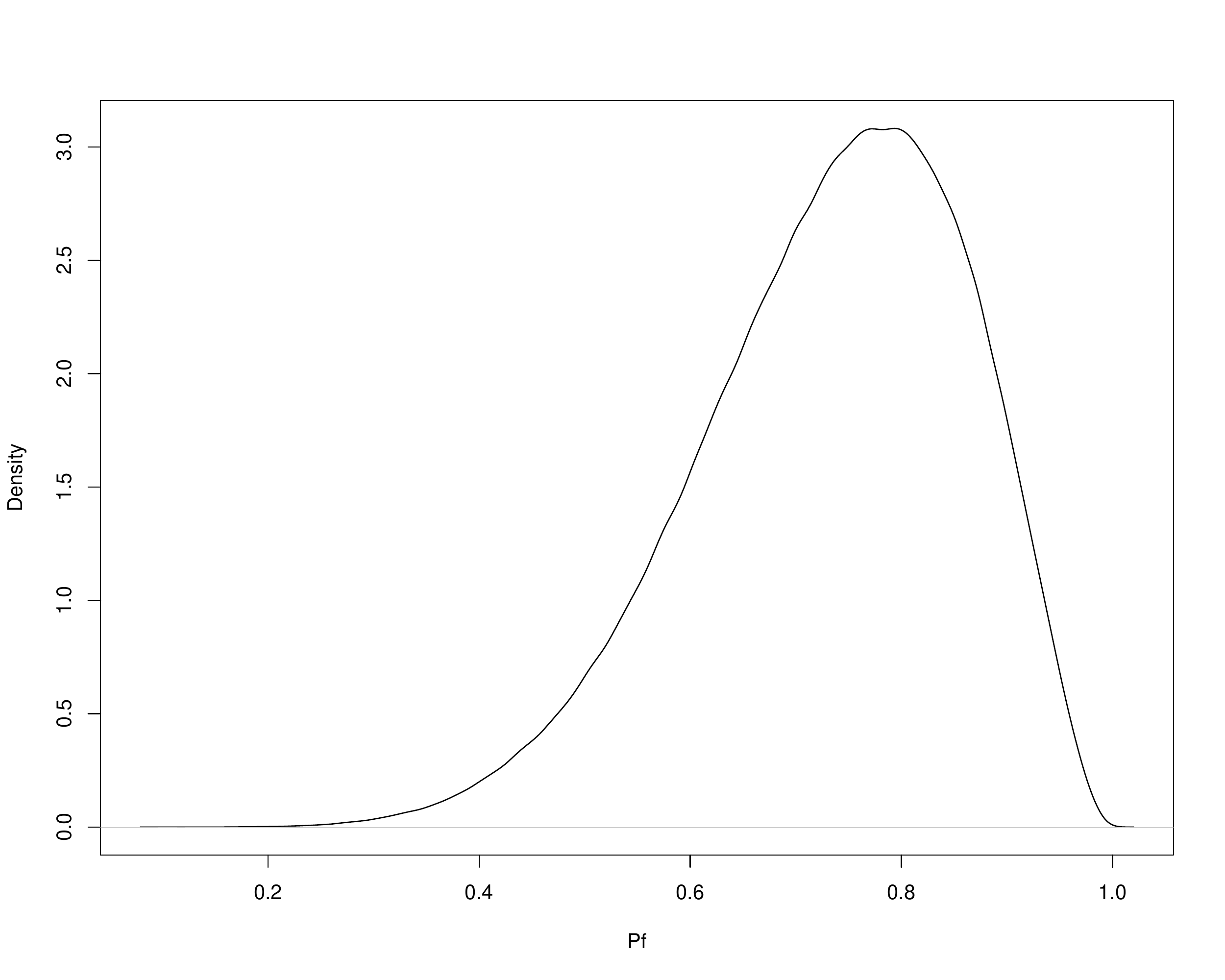}
\includegraphics[width = 3.8cm]{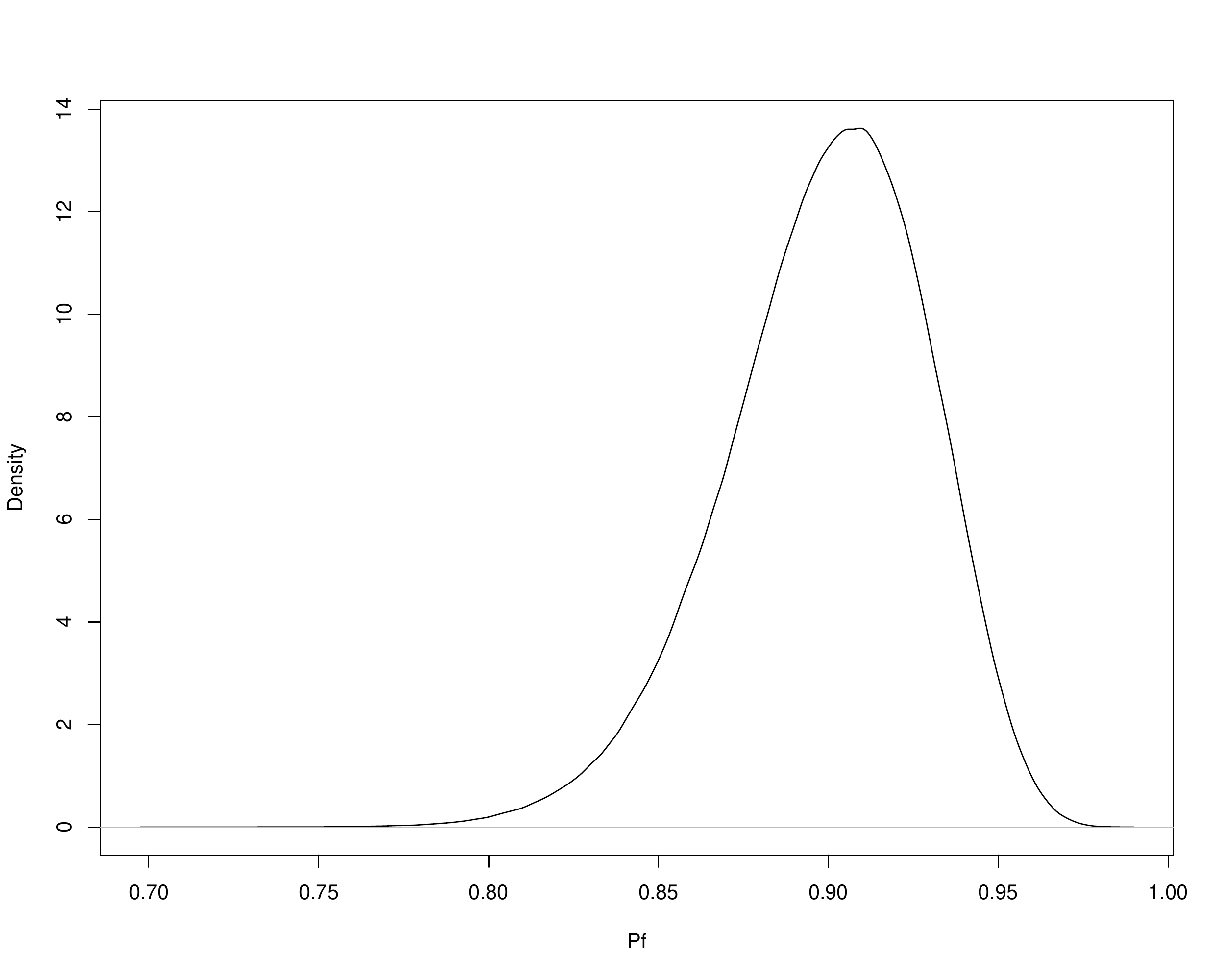}
 \includegraphics[width =3.8cm]{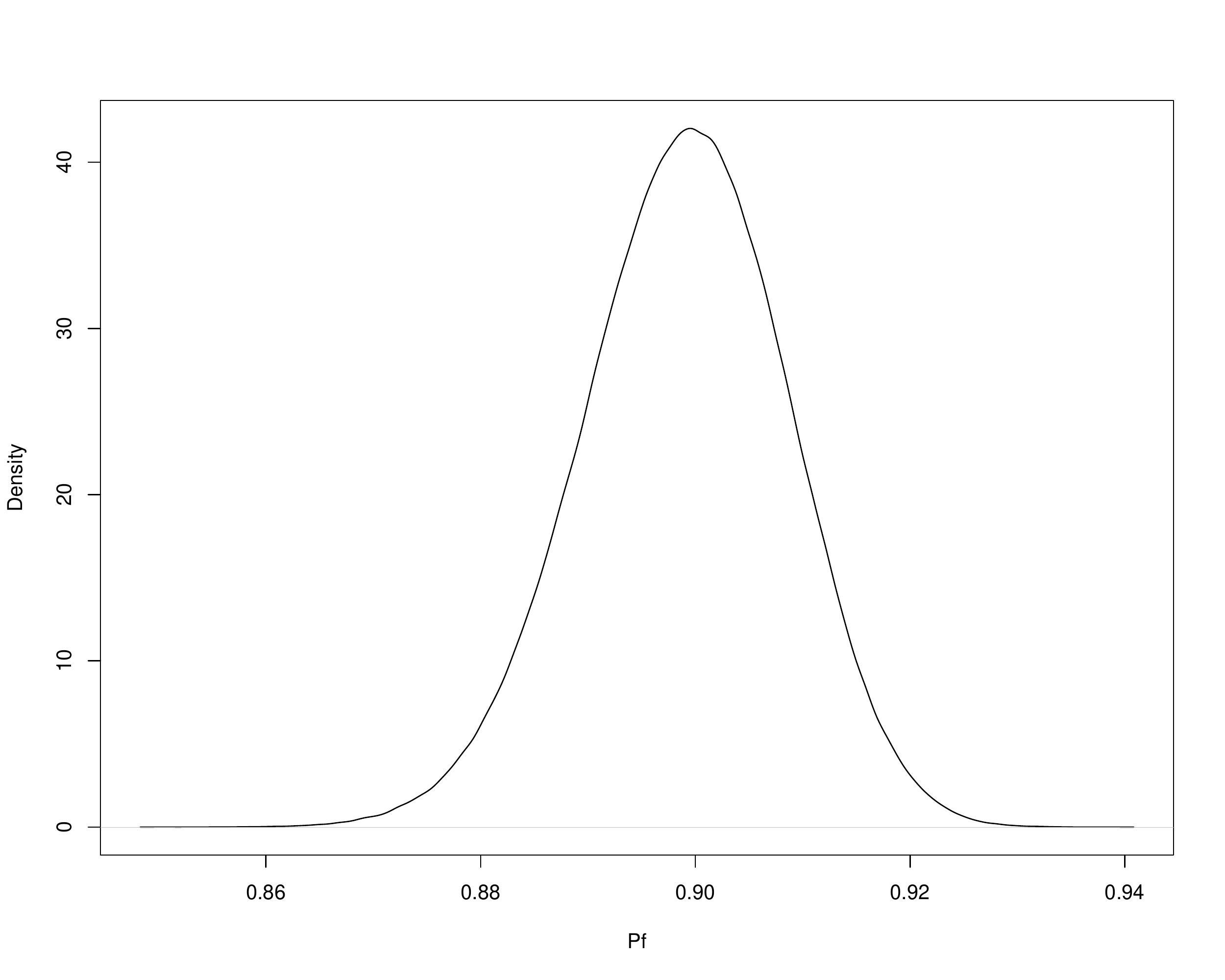}

\ \ \includegraphics[width =3.8cm]{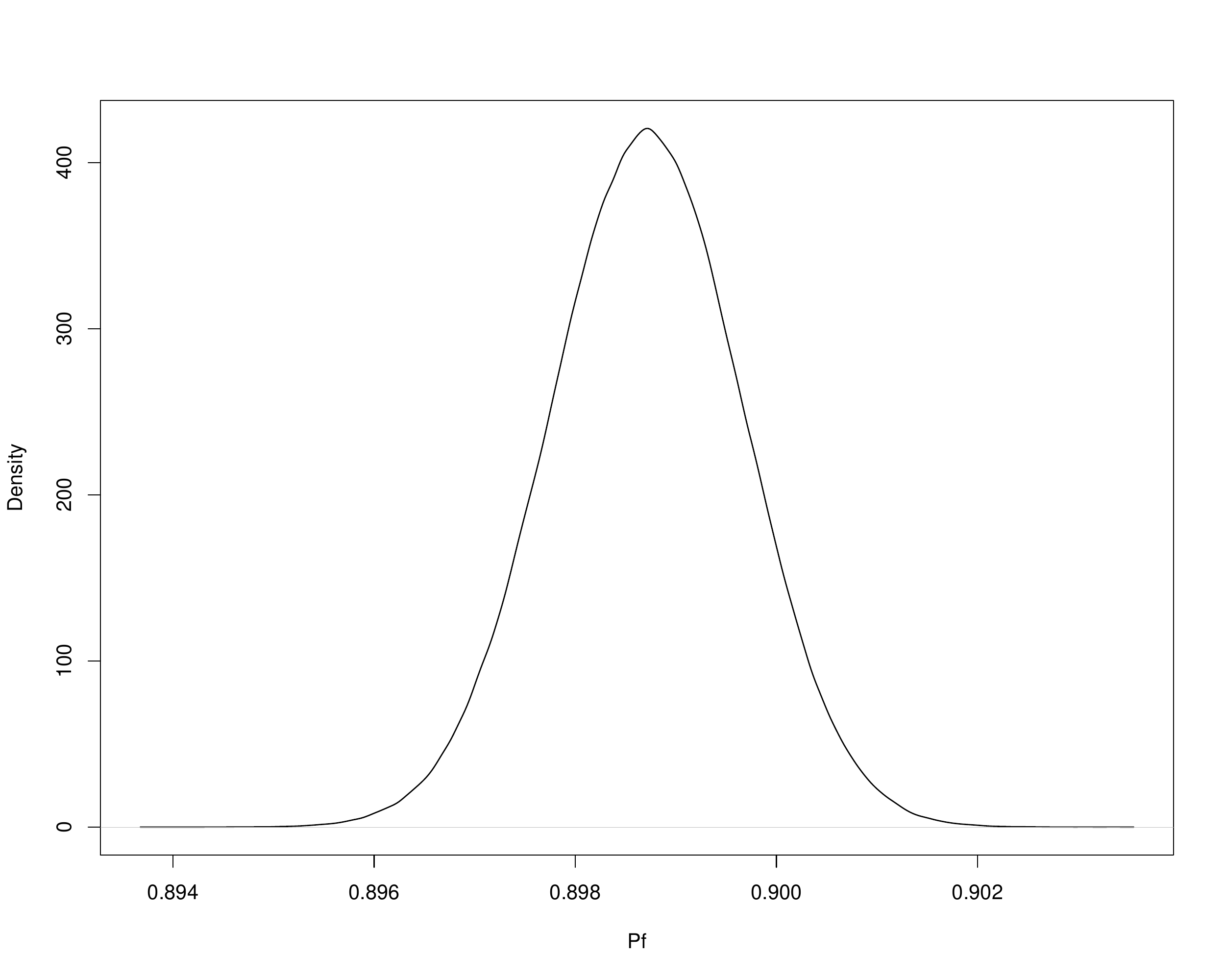}
\includegraphics[width = 3.8cm]{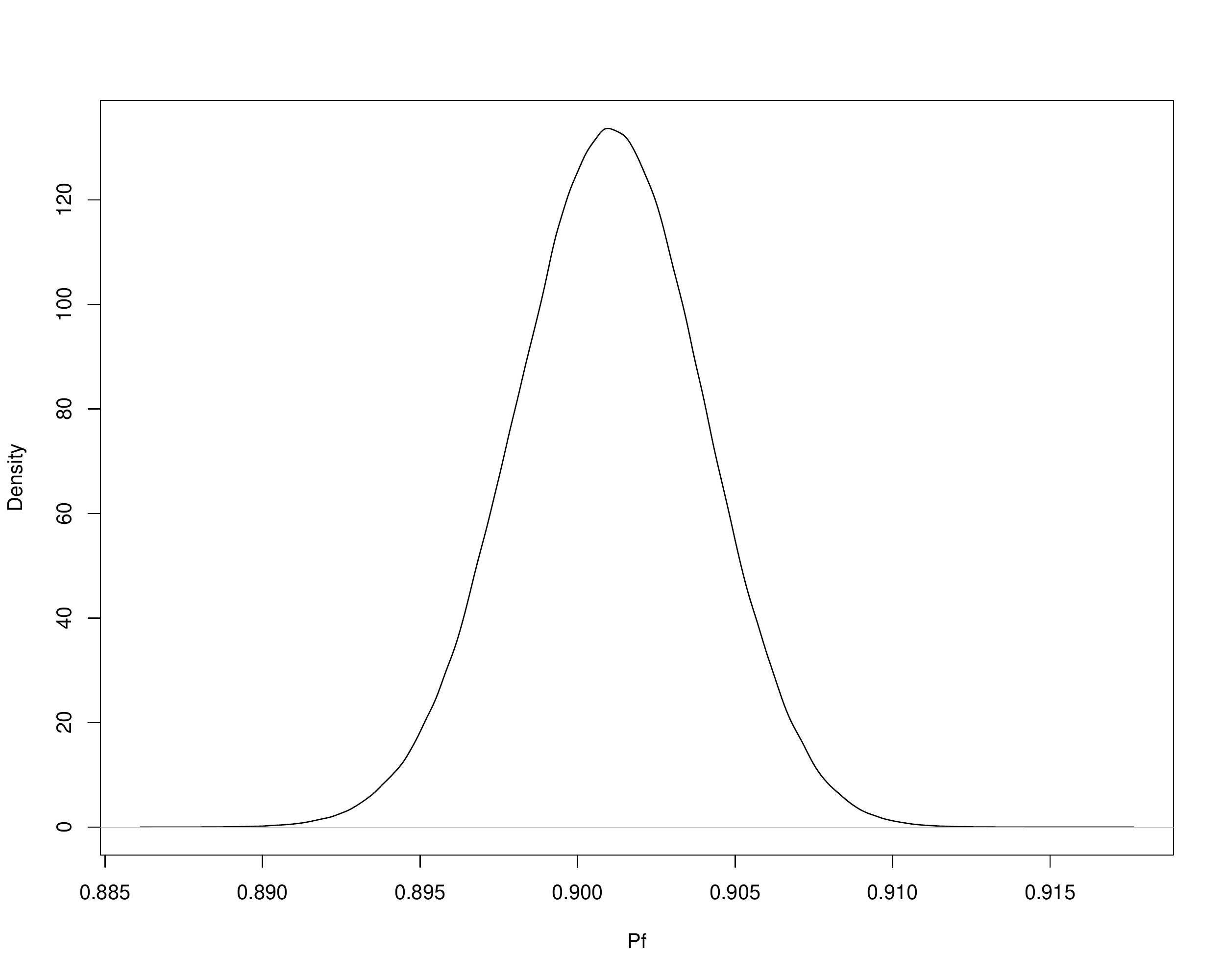}\hfill
\label{fig:densities}
\end{figure}

\section{Proofs}
\label{SectionProofs}
Let $\tilde{X}_1, \ldots, \tilde{X}_{K_n}$ be the distinct values in $X_1, \ldots, X_n$,
and let $N_{1,n}, \ldots, N_{K_n,n}$ be their multiplicities.
By Corollary~20 in \cite{Pitman1996b} (or see \cite[Theorem 14.37]{FNBI}), 
the posterior distribution of the Pitman-Yor process can be characterised as  the distribution of 
\begin{equation}
\label{EqPY}
\PY_n= R_nS_n+(1-R_n)Q_n, 
\end{equation}
for 
\begin{equation}
\label{EqDefSn}
S_n=\sum_{i=1}^{K_n} W_{n,i}\delta_{\tilde X_i},
\end{equation}
and independent variables $R_n,W_n,Q_n$ with, conditionally on $X_1,\ldots, X_n$, distributed
according to:
\begin{itemize}
\item $R_n\sim \Beta(n-\sigma K_n, M+\sigma K_n)$,
\item $Q_n\sim \PY(\sigma, M+ \sigma K_n ,G)$,
\item $W_n=(W_{n,1},\ldots, W_{n,K_n})\sim \Dir(K_n;N_{n,1}-\sigma,\ldots, N_{n,K_n}-\sigma)$.
\end{itemize}
Here $\Beta$ and $\Dir$ refer to the beta and Dirichlet distributions, respectively.
The number $K_n$ will tend almost surely to the total number of atoms of $P_0$ in the case
that $P_0$ is finitely discrete, and it will tend to infinity otherwise. In the latter case the 
rate of growth can have any order $n^\gamma$, for $0<\gamma\le 1$. (See
Theorem~8  of \cite{Karlin1967}, where it is shown that   $K_n/\EE K_n\rightarrow 1$, almost surely, where any rate can occur for $\EE K_n$.) 
The  proofs below use that $K_n/n$ tends to the mass $\lambda$ of the continuous part of $P_0$,
and the limit of the related sequence $n^{-1}K_n\tilde\PP_nf=n^{-1}\sum_{i=1}^{K_n}f(\tilde X_i)$.

\begin{lemma}
\label{LemmaConvergenceKn}
The number $K_n$  of distinct values among $X_1,\ldots,X_n\iid \lambda P_0^c+(1-\lambda)P_0^d$ satisfies $K_n/n\rightarrow \lambda$,
almost surely. The number $K_n^d$ of those values that belong to the set $S$ of atoms of $P_0^d$ satisfies $K_n^d/n\rightarrow 0$,
almost surely.
\end{lemma}

\begin{proof}
The number of distinct values not in $S$ is $K_n^c:=n\PP_n(S^c)$ and hence $K_n^c/n\rightarrow P_0(S^c)=\lambda$, 
almost surely. If $S=\{x_1,x_2,\ldots\}$,
then the number of distinct values in $S$ is bounded above by $m+n\PP_n\{x_{m+1},x_{m+2},\ldots\}$,  for any $m$,
and hence $K_n^d/n\le m/n+\PP_n\{x_{m+1},x_{m+2},\ldots\}\rightarrow P_0\{x_{m+1},x_{m+2},\ldots\}$,
almost surely, for every $m$.
\end{proof}

A class $\F$ of measurable functions $f: \X\to\RR$ is \emph{$P_0$-Glivenko-Cantelli} if the uniform law of large numbers
holds: $\sup_{f\in\F}|\PP_nf-P_0f|\ra0$, outer almost surely (see e.g.\ \cite{WCEP}, Chapter~2.4; we write ``outer'' because the supremum
may not be measurable; for standard examples this is superfluous).
An \emph{envelope function} of $\F$ is  a measurable function $F:\X\to\RR$ such that $|f|\le F$, for every $f\in\F$.

\begin{lemma}
\label{LemmaLLNDistinctSc}
Suppose $\F$ has an envelope function with $P_0F<\infty$.
If $S$ is the set of atoms of $P_0^d$, then $\sup_{f\in\F}  \bigl|  n^{-1}  \sum_{i=1}^{K_n}  f(\tilde X_i)1_{\tilde X_i\in S}  \bigr|  \rightarrow 0$, outer almost surely.
Furthermore, if $\s_n\ra \s\in [0,1]$, and $\F$ is a $P_0$-Glivenko-Cantelli class, then uniformly in $f\in\F$, outer almost surely,
$$\PP_nf+\frac{\sigma_n K_n}n\tilde\PP_nf\ra Tf:=(1-\lambda)P_0^d f+(1-\s)\lambda P_0^cf.$$
\end{lemma}

\begin{proof}
For any $M$, the supremum is bounded above by $n^{-1}K_n^d M+\PP_n F1_{F>M}\rightarrow
P_0F1_{F>M}$, almost surely. The first term tends to zero by Lemma~\ref{LemmaConvergenceKn}, for any $M$.
The second term can be made arbitrarily small by choosing $M$ large.

For the convergence in the display we write $K_n\tilde\PP_nf=\sum_{i=1}f(X_i)1_{X_i\notin S}+\sum_{i=1}^{K_n}f(\tilde X_i)1_{\tilde X_i\in S}$.
By the first assertion, the second sum divided by $n$ tends to zero, uniformly in $f$. The first sum divided by
$n$ tends to $\lambda P_0^cf$, where the convergence is uniform in $f\in\F$ if $\F$ is a Glivenko-Cantelli class (which implies
that the set of functions $x\mapsto f(x)1_{S^c}(x)$ is a Glivenko-Cantelli class, in view of \cite{vdVWPreservation}).
Thus the left side of the display tends to $P_0f+ \s\l P_0^cf$, which is equal to $Tf$.
\end{proof}


\subsection{Proof of  Theorem~\ref{thm:PYBVM}}
\label{SectionProofMain}
The left side  $\sqrt n\bigl(\PY_n-\PP_n+(\sigma K_n/n)(\tilde\PP_n-G)\bigr)$ of the theorem can be decomposed as
\begin{align}
\sqrt n\Bigl( R_n-1+ \frac{\sigma K_n}n\Bigr)(S_n-Q_n)
&+\sqrt n\biggl(S_n\Bigl(1-\frac{\sigma K_n}n\Bigr)-\PP_n+\frac{\sigma K_n}n\tilde\PP_n\biggr)\nonumber\\
&\qquad+\sigma\sqrt{K_n}(Q_n-G)\sqrt{\frac{K_n}n}.\label{EqDecomp}
\end{align}
We derive the limit distributions of these three terms in Lemmas~\ref{LemOne}--\ref{LemThree} below. For later use it will
be helpful to allow $\s\in(0,1)$ to depend on $n$.
For this reason we give precise proofs of the first two lemmas,
although they are very similar to results obtained in \cite{James2008, FNBI}. The main novelty is in
the third lemma. For simplicity we assume that $\s_n\in(0,1)$ converges to a limit, which we allow to be 0 or 1. 

\begin{lemma}
\label{LemOne}
If $\s_n\ra \s\in [0,1]$, then
\begin{align}\label{EqConvergenceR}
\sqrt n\left(  R_n-1+ \frac{\sigma_n K_n}n\right)\given X_1,\ldots, X_n
&\weak N\bigl(  0, (1-\sigma \lambda)\sigma\lambda\bigr),\qquad\text{a.s.}
\end{align}
\end{lemma}

\begin{proof}
We can represent the beta variable $R_n$ as the quotient $R_n=U_n/(U_n+V_n)$, for independent gamma variables 
$U_n\sim \Gamma(u_n,1)$ and $V_n\sim \Gamma(v_n,1)$, for $u_n=n-\s_n K_n$ and $v_n=M+\s_nK_n$
the means, and also variances, of the latter variables. We can decompose
$$(U_n+V_n)\Bigl(R_n-\frac{u_n}{u_n+v_n}\Bigr)=\frac{v_n}{u_n+v_n}(U_n-u_n)-\frac{u_n}{u_n+v_n}(V_n-v_n).$$
Since $\s_nK_n/n\ra \s\lambda\in[0,1]$, we have $v_n/(u_n+v_n)\ra\s\lambda$ and $u_n/(u_n+v_n)\ra 1-\s\l$.
Furthermore, $(U_n+V_n)/n\ra 1$, almost surely, by the law of large numbers.

If $\s\l<1$, then $n-\s_nK_n\ra\infty$ and hence $(U_n-u_n)/\sqrt {u_n}\weak Z_1\sim N(0,1)$, by the central limit theorem.
It follows that $(U_n-u_n)/\sqrt n\weak Z_1\sqrt {1-\s\l}$. If $\s\l=1$, then $\var (U_n/\sqrt n)=u_n/n\ra0$ and hence
$(U_n-u_n)/\sqrt n\weak 0$, where the limit 0 is identical to $Z_1\sqrt {1-\s\l}$ in this case. Thus
in all cases $(U_n-u_n)/\sqrt n\weak Z_1\sqrt {1-\s\l}$.

If $\s\l>0$, then $\s_nK_n\ra\infty$ and hence $(V_n-v_n)/\sqrt {v_n}\weak Z_2\sim N(0,1)$, by the central limit theorem.
It follows that $(V_n-vn)/\sqrt n\weak Z_2\sqrt {\s\l}$. If $\s\l=1$, then $\var (V_n/\sqrt n)=v_n/n\ra0$ and hence
$(V_n-v_n)/\sqrt n\weak 0$, where the limit 0 is identical to $Z_2\sqrt {\s\l}$ in this case. Thus
in all cases $(V_n-v_n)/\sqrt n\weak Z_2\sqrt {\s\l}$.

Combining the preceding, we see that the sequence $\sqrt n\bigl(R_n-u_n/(u_n+v_n)\bigr)$ converges weakly to
$\s\l Z_1\sqrt{1-\s\l}+(1-\s\l)Z_2\sqrt{\s\l}$.
As the limit variable has variance $(1-\sigma \lambda)\sigma\lambda$ and $u_n/(u_n+v_n)=(1-\s_nK_n/n)(1+O(1/n))$, 
this concludes the proof.
\end{proof}


\begin{lemma}
\label{LemTwo}
If $\s_n\ra \s\in[0,1]$ and $K_n\ra\infty$ and $\F$ is a class of finitely many $G$-square-integrable functions, then 
in $\RR^\F$,
\begin{align}
\label{EqConvergenceQ}
\sigma _n\sqrt{K_n}(Q_n-G)\given X_1,\ldots, X_n&\weak \sqrt{\s(1-\sigma)}\,\GG_G.\qquad \hbox{a.s.}
\end{align}
The convergence is also true in $\linfty(\F)$ if $\F$ possesses a $G$-square integrable envelope function
and the Pitman-Yor process $\PY(\s,\s,G)$ satisfies the central limit theorem in this space.
\end{lemma}

\begin{proof}
The  process  $Q_n\sim \PY(\s_n, M+\s_nK_n, G)$ centered at mean zero can be represented as 
$$Q_n-G\sim\sum_{i=0}^{K_n}W_{n,i}(P_i-G),$$ where $(W_{n,0},\ldots, W_{n,K_n})\sim \Dir(K_n+1;M,\s_n,\ldots,\s_n)$ is independent
of the independent processes $P_0\sim \PY(\s_n, M,G)$ and $P_i\iid\PY(\s_n,\s_n,G)$, for $i=1,\ldots K_n$ (see e.g.\
Proposition~14.35 in \cite{FNBI}).
The variable $W_{n,0}$ is $B(M,K_n\s_n)$-distributed, whence
$$\s_n\sqrt{K_n}\EE\bigl|W_{n,0}(P_0-G)f\bigr|= \frac{\s_n\sqrt{K_n}M}{M+K_n\s_n}\EE|(P_0-G)f|\le \frac{M}{\sqrt{K_n}} \sqrt{Gf^2},$$
where the moment of $(P_0-G)f$ can be obtained from Proposition~14.34 in \cite{FNBI}.
Next by the gamma representation of the Dirichlet distribution (e.g.\ Propositions~G.2 and G.3 in \cite{FNBI}), we can represent 
$$\s_n\sqrt{K_n}\sum_{i=1}^{K_n}W_{n,i}(P_i-G)\sim (1-W_{n,0})\frac{K_n^{-1/2}\sum_{i=1}^{K_n}V_{n,i}(P_i-G)}{K_n^{-1}\sum_{i=1}^{K_n}V_{n,i}/\s_n},$$
where the variables $V_{i,n}\iid \Gamma(\s_n,1)$ are independent of $W_{n,0}$ and the $P_i$.
The triangular array of variables $V_{n,i}(P_i-G)$ are i.i.d.\ for every $n$ with
\begin{align*}
&\EE V_{n,1}^2\bigl((P_1-G)f\bigr)^2=\s_n(1+\s_n)G(f-Gf)^2\frac{1-\s_n}{1+\s_n},\\
&\EE V_{n,1}^2\bigl((P_1-G)f\bigr)^21_{|V_{n,1}(P_1-G)f|\ge M_n}\ra0,
\end{align*}
for any $M_n\ra\infty$. The second claim is implied by the uniform integrability of the set of variables 
$W_\sigma := V_\sigma^2 \left( (P_\sigma - G) f \right)^2$,
for $\s\in [0,1]$, where $V_\sigma \sim \Gamma(\sigma, 1)$ is independent of $P_\sigma \sim \textrm{PY}(\sigma, \sigma, G)$,
and $W_0$ and $W_1$ are defined to be degenerate at 0, in agreement with the first line of the
preceding display. This
 itself is a consequence of the continuity of the map $\s \mapsto W_\s$ from $[0,1]$ to $L_2(\Omega)$ and the
Dunford-Pettis theorem. The continuity follows from the norm continuity, $\E W_{\s_n}^2\ra \E W_\s^2$, if $\s_n\ra \s$, 
by the first assertion
in the display, combined with the continuity in distribution of $\s\mapsto W_\s$.
Therefore, the sequence $K_n^{-1/2}\sum_{i=1}^{K_n}V_{n,i}(P_i-G)$ tends to a normal distribution
with mean zero and variance $\s (1-\s) G(f-Gf)^2$, by the Lindeberg central limit theorem. 
The linearity of the process in $f$ shows that as a process it tends marginally in distribution to
the process $\sqrt{\s(1-\sigma)}\,\GG_G$. 
Because $\var \bigl(K_n^{-1}\sum_{i=1}^{K_n}V_{n,i}/\s_n\bigr)=1/(K_n\s_n)$, we have 
$K_n^{-1}\sum_{i=1}^{K_n}V_{n,i}/\s_n\ra 1$, in probability, if $K_n\s_n\ra\infty$. Since also $1-W_{n,0}\ra1$,
the proof is complete in the case that $K_n\s_n\ra\infty$.

If $K_n\s_n$ remains bounded, then necessarily $\s_n\ra0$, as $K_n\ra\infty$, by assumption. Then
\begin{align*}
\s_n^2K_n \EE \Bigl(\sum_{i=0}^{K_n}W_{n,i}(P_i-G)f\Bigr)^2
&=\s_n^2K_n \sum_{i=0}^{K_n}\sum_{j=0}^{K_n}\EE W_{n,i}W_{n,j}(P_i-G)f(P_j-G)f\cr
&\le \s_n^2K_n  \EE \Bigl(\sum_{i=0}^{K_n}W_{n,i}\Bigr)^2Gf^2\le \s_n^2K_nGf^2.
\end{align*}
Since this tends to zero, the lemma holds also in this case, with a limit process equal to 0, which is 
equal to $\sqrt{\s(1-\s)}\GG_G$.

For the final assertion we note that the preceding argument gives
the convergence of $\sup_{f\in\F}\s_n\sqrt{K_n}W_{n,0}(P_0-G)f$ to zero
 for any class $\F$ with square-integrable envelope function. The convergence of
$K_n^{-1/2}\sum_{i=1}^{K_n}V_{n,i}(P_i-G)$ in $\infty(\F)$ follows from the convergence of $K_n^{-1/2}\sum_{i=1}^{K_n}(P_i-G)$
by the multiplier central limit theorem (e.g.\ Lemma~2.9.1 and Theorem~2.9.2 in \cite{WCEP}).
\end{proof}

\begin{lemma}
\label{LemThree}
If $\s_n\ra\s\in [0,1]$, where $\s\l<1$, then for any $P_0$-Donsker class with square-integrable envelope function
\begin{align}
\sqrt n\biggl(S_n\Bigl(1-\frac{\sigma_n K_n}n\Bigr)-\PP_n+\frac{\sigma K_n}n\tilde\PP_n\biggr)\weak 
\WW -\frac1{1-\s\l}\WW1\,T,\qquad\text{a.s.},
\label{EqConvergenceWprocess}
\end{align}
in $\linfty(\F)$, where $\WW=\sqrt{\l(1-\s)}\GG_{P_0^c}+\sqrt{1-\l}\GG_{P_0^d}$, for independent Brownian bridge processes
$\GG_{P_0^c}$ and $\GG_{P_0^d}$, and $T$ is the (deterministic) process defined in Lemma~\ref{LemmaLLNDistinctSc}.
The convergence is true in probability for any $P_0$-Donsker class.
If $\s_n\ra\s\in [0,1]$, where $\s\l=1$, then the sequence of processes tends to the zero proces.
\end{lemma}

\begin{proof}
A gamma representation for the multinomial vector $W_n$ in the definition of $S_n$ is
\[W_{n,i}=\frac{  U_{i,0}  +  \sum_{j=1}^{N_{n,i} - 1} U_{i,j}}{  \sum_{i=1}^{K_n}  \left(    U_{i,0}    +    \sum_{j=1}^{N_{n,i}-1} U_{i,j}  \right)},\]
for all $U_{i,j}$ independent, $U_{i,0}\sim \Gamma(1-\sigma,1)$ and $U_{i,j}\sim \Gamma(1,1)$, for $j\ge 1$.
Relabel the $n$ variables $U_{i,j}$ as $\xi_{n,1},\ldots,\xi_{n,n}$, as follows. Let $S$ be the set of all atoms
of $P_0$. An observation $X_i$  that is not contained in $S$ appears exactly once
in the set $\{X_1,\ldots, X_n\}$ of observations; set the variable $\xi_{n,i}$ with the corresponding $i$ equal to $U_{i,0}$.
Every $X_i$ that is contained in $S$ appears $N_{n,i}\ge 1$ times among $X_1,\ldots, X_n$; set the $\xi_{n,j}$ with indices
corresponding to these appearances equal to $U_{i,0},U_{i,1},\ldots, U_{i,N_{n,i}-1}$. Then
\begin{align}
\label{EqQuotient}
S_n=\sum_{i=1}^{K_n} W_{n,i} f(\tilde X_i)=
\frac{  n^{-1}\sum_{i=1}^n \xi_{n,i}f(X_i)}{  n^{-1}\sum_{i = 1}^n \xi_{n,i}}=:\frac{\overline S_nf}{\overline S_n1},
\end{align}
and the left side of the lemma can be decomposed as 
\begin{align*}
&S_nf\,\sqrt n\Bigl(1-\frac{\s_nK_n}n-\overline S_n1\Bigr)+\sqrt n\Bigl(\overline S_nf-\PP_nf+\frac{\s_n K_n}n\tilde\PP_n f\Bigr)\\
&\qquad=-S_nf\sqrt n (\overline S_n1-T_n1)+\sqrt n(\overline S_nf-T_nf),
\end{align*}
where $T_nf=\PP_nf-(\s_nK_n/n)\tilde\PP_nf$ tends to $Tf$, by Lemma~\ref{LemmaLLNDistinctSc}.
We shall show that $\sqrt n (\overline S_n-T_n)\given X_1,\ldots, X_n\weak \WW$. Then $S_nf\ra Tf/T1=Tf/(1-\s\l)$, and the
result follows in the case that $\s\l<1$.

The variables $\xi_{n,1},\ldots,\xi_{n,n}$ are independent. The $K_n$ variables corresponding to 
the distinct values are $\Gamma(1-\sigma,1)$-distributed;
the others are $\Gamma(1,1)$-distributed. 
Thus the conditional mean and variance of $\overline S_nf$ are given by 
\begin{align*}
\sum_{i = 1}^n (\EE \xi_{n,i}) f(X_i)&=\sum_{i = 1}^n f(X_i)-\sigma\sum_{i=1}^{K_n}f(\tilde X_i)=T_nf,\\
\frac{1}{n}\sum_{i = 1}^n (\var\xi_{n,i})f^2(X_i)&=\frac{1}{n}\sum_{i = 1}^n f^2(X_i)-\frac{\sigma}{n}\sum_{i=1}^{K_n} f^2(\tilde X_i)
\ra Tf^2,\qquad \text{a.s.},
\end{align*}
by Lemma~\ref{LemmaLLNDistinctSc}. The limit variance is equal to $\var \WW f$.
To complete the proof of the convergence $\sqrt n (\overline S_n-T_n)f\given X_1,\ldots, X_n\weak \WW$,
it suffices to verify the Lindeberg-Feller condition. We have, for $\xi_n\sim\Gamma(1-\s_n,1)$ and $\bar \xi_n\sim \Gamma(1,1)$,
\begin{align*}
& \frac{1}{n}  \sum_{i = 1}^n   \EE  \left(    \xi_{n,i}^2 f^2(X_i)1_{|\xi_{n,i} f(X_i)|>\eps\sqrt n}\given  X_1,\ldots, X_n \right)\\
&\qquad\le \max\Bigl(\EE \xi_{n}^21_{|\xi_{n}|\,\max_{1\le i\le n}|f(X_i)|>\e\sqrt n},\EE \bar\xi_{n}^21_{|\bar \xi_{n}|\,\max_{1\le i\le n}|f(X_i)|>\e\sqrt n}\Bigr)\, \PP_nf^2.
\end{align*}
This tends to zero for every sequence $X_1,X_2,\ldots$ such that both $\PP_nf^2=O(1)$ and $\max_{1\le i\le n}|f(X_i)|/\sqrt n\rightarrow0$, which is
almost every sequence if $P_0f^2<\infty$.

By the Cram\'er-Wold device and linearity in $f$, the convergence is then implied 
for finite sets of $f$. 

For convergence as processes in $\linfty(\F)$ for a general Donsker class, it suffices to prove asymptotic tightness
(see e.g.\ Theorem~1.5.4 in \cite{WCEP}).
The processes $n^{-1/2}\sum_{i = 1}^n (\xi_{n,i}-\EE \xi_{n,i})f(X_i)$ are multiplier processes with
mean zero, independent multipliers. Because the multipliers are not i.i.d., a direct 
application of the conditional multiplier central limit theorem (see Theorem~2.9.7 in \cite{WCEP}) is not possible.
However, the multipliers have two forms $\Gamma(1-\sigma,1)$ and $\Gamma(1,1)$. By
Jensen's inequality, for any collection $\G$ of functions,
\[
\EE_\xi
  \left\|
    \sum_{i = 1}^n 
      (\xi_{n,i}
      -
      \EE \xi_{n,i})f(X_i)
  \right\|_{\G}^*
  \le 
\EE_{\xi,\xi'}
  \left\|
    \sum_{i = 1}^n 
    \left(
      \xi_{n,i}
      -
      \EE\xi_{n,i}
      +
      \xi_{n,i}'
      -
      \EE\xi_{n,i}'
    \right)
    f(X_i)
  \right\|_{\G}^*,
\]
for any random variables $\xi_{n,i}'$ independent of the $\xi_{n,i}$. We can choose these variables so that
all $\xi_{n,i}+\xi_{n,i}\iid\Gamma(1,1)$. The process in the right side then does have i.i.d.\ multipliers,
and the asymptotic tightness follows from the i.i.d.\ case (as in \cite{WCEP}), Theorems~3.6.13, 2.9.6 and~2.9.7; also see 
Corollary~2.9.9; we apply the preceding inequality with $\G$ equal to the set of differences
$f-g$ of functions $f,g\in\F$ with $L_2(P_0)$-norm of $f-P_0f-g+P_0g$ smaller than $\delta$).

Finally if $\s\l=1$, then both $\s=1$ and $\l=1$. The second implies that $P_0=P_0^c$, $K_n=n$ and $\tilde\PP_n=\PP_n$.
Thus in this case $S_n(1-\s_nK_n/n)=\sumin W_{n,i}f(X_i)(1-\s_n)$, for 
$(W_{n,1},\ldots,W_{n,n})\sim \Dir(n,1-\s_n,\ldots, 1-\s_n)$, and $T_nf=\PP_nf(1-\s_n)$.. We can now compute
\begin{align*}
\EE\Bigl(\sumin W_{n,i}f(X_i)\given X_1,\ldots, X_n\Bigr)&=\PP_nf,\\
\var\Bigl(\sumin W_{n,i}f(X_i)\given X_1,\ldots, X_n\Bigr)&=\sum_{i=1}^n\sum_{j=1}^n \cov(W_{n,i},W_{n,j})f(X_i)f(X_j)\\
&\le \sum_{i=1}^n\frac{(n-1)f^2(X_i)}{n^2(n(1-\s_n)+1)}\le \frac{\PP_nf^2}{n(1-\s_n)},
\end{align*}
as the covariances between the $W_{n,i}$ are negative. This implies that
the conditional mean and variance of $\sqrt n (S_n(1-\s_n)-T_nf)$ tend to zero, as $\s_n\ra 1$.
\end{proof}

We are ready to complete the proof of Theorem~\ref{thm:PYBVM}.  If $K_n\ra\infty$,
then Lemmas~\ref{LemThree}--\ref{LemTwo} together with the convergence $K_n/n\ra\l$
immediately give the convergence of
the second and third terms in the decomposition \eqref{EqDecomp}. Furthermore, these
lemmas give that $S_n-Q_n\ra Tf/(1-\s\l)$, which combined with Lemma~\ref{LemOne}
gives the convergence of the first term in \eqref{EqDecomp}.

If $K_n$ remains bounded, then Lemma~\ref{LemTwo} does not apply. However,
since the process $Q_n$ will run through finitely many different Pitman-Yor processes,
we have $Q_n-G=O_P(1)$ and hence the third term in \eqref{EqDecomp} is $O_P(1/\sqrt n)$, still under
the assumption that $K_n$ is bounded. Lemma~\ref{LemThree}  is still valid,
and hence the second term  in \eqref{EqDecomp} converges to a Gaussian process as before.
We can divide this term by $1-\s K_n/n\ra 1$, to see that $S_n\ra T$, in view of Lemma~\ref{LemmaLLNDistinctSc}.
The sequence $K_n$ can remain bounded only if  
$\lambda=0$ and then the normal limit in 
Lemma~\ref{LemOne} is degenerate, whence $\sqrt n (R_n-1)=-\sigma K_n/\sqrt{n}+o_P(1)=o_P(1)$, almost surely,
again under the assumption that $K_n$ is bounded. Combined this shows that the first
term in  \eqref{EqDecomp}  tends to zero.

\subsection{Proof of  Theorem~\ref{thm:PYBVMEstimatedSigma}}
\label{Sectionthm:PYBVMEstimatedSigma}
Make the dependence  on $\s$ of the Pitman-Yor posterior process and its limit   explicit 
by writing $\PY_n(\s)$ and $\GG(\s)$ for the process $\PY_n$ in \eqref{EqPY} and the right side in 
Theorem~\ref{thm:PYBVM}, and set
$$\CPY_n(\s)=\sqrt n \bigl(\PY_n(\s)-\PP_n-\frac{\s K_n}n(G-\tilde \PP_n)\bigr).$$
Lemmas~\ref{LemOne}--\ref{LemThree} give
\begin{equation}
\label{EqUniformityInSigma}
\sup_{\s\in (0,1)}\sup_{h\in \text{BL}_1}\Bigl|\EE \Bigl(h\bigl(\CPY_n(\s)\bigr)\given X_1,\ldots, X_n\Bigr)-\EE h\bigl(\GG(\s)\bigr)\Bigr|\ra 0,
\end{equation}
in probability. This immediately gives that for every data-dependent  $\hat\s_n$ that take their values in the interval $(0,1)$,
$$\sup_{h\in \text{BL}_1}\Bigl|\EE \Bigl(h\bigl(\CPY_n(\hat\s_n)\bigr)\given X_1,\ldots, X_n\Bigr)-\EE h\bigl(\GG(\hat\s_n)\bigr)\Bigr|\ra 0,$$
in probability, where the second expectation is on the limit process $\GG(\hat\s_n)$ for given, fixed $\hat\s_n$.
The continuity of the limit process in $\s$ shows that, for $\hat\s_n\ra \s_0$ in probability,
$$\sup_{h\in \text{BL}_1}\Bigl|\EE h\bigl(\GG(\hat\s_n)\bigr)-\EE h\bigl(\GG(\s_0)\bigr)\Bigr|\ra 0,$$
in probability. Combined the two preceding displays give the first assertion of Theorem~\ref{thm:PYBVMEstimatedSigma}.

For discrete $P_0$ with regularly varying atoms, the convergence of the maximum likelihood estimator
$\hat\s_n$ to its coefficient of regular variation $\s_0\in (0,1)$ is shown in Theorem~\ref{ThmSigma}, and hence
the preceding argument applies.

In a hierarchical Bayesian setup with a prior on $\s$ and given $\s$ the Pitman-Yor prior on $P$, the posterior distribution of $P$ can be decomposed as
\begin{align*}
&\EE \bigl(h\bigl(\CPY_n(\s)\given X_1,\ldots,X_n\bigr)\\
&\qquad=  \int \EE \Bigl(h\bigl(\CPY_n(\s)\bigr)\given \s, X_1,\ldots,X_n\Bigr)\,\Pi_n(d\s\given X_1,\ldots, X_n), 
\end{align*}
where $\Pi_n(d\s\given X_1,\ldots, X_n)$ refers to the posterior distribution of $\s$ given the observations $X_1,\ldots, X_n$,
and $\CPY_n(\s)\given \s, X_1,\ldots,X_n$ is the standardised Pitman-Yor posterior distribution for given $\s$, considered in Theorem~\ref{thm:PYBVM}.
The uniformity \eqref {EqUniformityInSigma} shows that the expectation in the integral on the right side can be replaced
asymptotically by $\EE h\bigl(\GG(\s)\bigr)$, uniformly in $h\in\text{BL}_1$, whenever the posterior distribution of 
$\s$ concentrates  with probability tending to one on the interval $(0,1)$. In particular,
this is true if the posterior distribution of $\s$ is consistent for some value $\s_0\in(0,1)$, i.e.\ if it concentrates asymptotically within
the interval $(\s_0-\e,\s_0+\e)$, for every $\e>0$. This consistency is shown in the proposition below.
Given posterior consistency, by the continuity of the limit process in $\s$, the 
expectation $\EE h\bigl(\GG(\s)\bigr)$ can in turn in the limit be replaced by $\EE h\bigl(\GG(\s_0)\bigr)$, 
uniformly in $h\in\text{BL}_1$.
This gives the second assertion of Theorem~\ref{thm:PYBVMEstimatedSigma}.

\subsection{Estimating the type parameter}
\label{SectionEstimatingType}
A measurable function $\a: [1,\infty)\to \RR_+$ is said to be \emph{regularly varying} (at $\infty$) of order $\g$ if, for all $u>0$, as $n\ra\infty$,
\begin{equation}
\label{EqRV}
\frac{\a(nu)}{\a(n)}\ra u^\g.
\end{equation}
It is known (see e.g.\ \cite{Binghametal} or the appendix to \cite{deHaan})
that if the limit of the sequence of quotients on the left exists for every $u$, then it necessarily has the form $u^\g$, for some $\g$, as in \eqref{EqRV}.
If we write $\a(u)=u^\g L(u)$, then $L$ will be \emph{slowly varying}: a function that is regularly varying of order 0.
Then $\a(n)=n^\g L(n)$, and it can be shown that $n^{\g-\d}\ll \a(n)\ll \a^{n+\d}$, for every $\d>0$, so that the
rate of growth of $\a$ is $n^\g$ to ``first order''. (See Potter's theorem, \cite{Binghametal}, Theorem 1.5.6, or \cite{deHaan}, Proposition~B.1.9-5).

\begin{example}
For the probability distribution $(p_j)_{j\in\NN}$ with $p_j=C/j^\a$, for some $\a>1$, the function
$\a(u):=\# (j: 1/p_j\le u)= \lfloor (Cu)^{1/\a}\rfloor$ is regularly varying of order $\g=1/\a$.
\end{example}

We consider the empirical Bayes estimator $\hat\s_n$, the maximum likelihood estimator in the model
$P\given \s\sim\PY \left( \sigma, M, G \right)$ and $X_1, \ldots, X_n \given P,\s \sim P$ given observations $X_1,\ldots, X_n$.
We also consider the posterior distribution of $\s$ given $X_1,\ldots, X_n$ in the model 
$\s\sim \Pi_\s$, $P\given \s\sim\PY \left( \sigma, M, G \right)$ and $X_1, \ldots, X_n \given P,\s \sim P$, for
a given prior distribution $\Pi_\s$ on $(0,1)$. In both cases the likelihood for observing $X_1,\ldots, X_n$
 is proportional to \eqref{EqLikelihoodSigma}. Hence $\hat\s_n$ is
the point of maximum of this function and, by Bayes theorem, the posterior distribution has density relative to $\Pi_\s$ proportional to 
\eqref{EqLikelihoodSigma}. 

In the following theorem we consider these objects under the assumption that  $X_1,\ldots,X_n$ are an i.i.d.\ sample
from a distribution $P_0$. Consistency of $\hat\s_n$ for $\s_0$ means that $\hat\s_n\ra\s_0$ in probability. Contraction
of the posterior distribution to $\s_0$ means that $\Pi_n(\s: |\s-\s_0|>\e\given X_1,\ldots, X_n)$ tends to zero in probability,
for every $\e>0$.

\begin{theorem}
\label{ThmSigma}
If $P_0$ is discrete with atoms such that $\a_0(u):=\#\{ x: 1/P_0\{x\}\le u\}$
is regularly varying of exponent $\s_0\in(0,1)$, then  the empirical Bayes estimator  $\hat\s_n$ is consistent for $\s_0$.
Furthermore, for a prior distribution $\Pi_\s$  on $\s$ with a density that is bounded away from zero and infinity, 
the posterior distribution of $\s$ contracts to $\s_0$.
\end{theorem}

\begin{proof}
Up to an additive constant the log likelihood can be written
\begin{align*}
\Lambda_n(\s)&=\sum_{l=1}^{K_n-1}\log (M+l\s)+\sum_{j=1: N_{n,j}\ge 2}^{K_n}\sum_{l=0}^{N_{n,j}-2}\log (1-\s+l)\\
&=\sum_{l=1}^{K_n-1}\log (M+l\s)+\sum_{l=1}^{n-1}\log (l-\s)Z_{n,l+1},
\end{align*}
where $Z_{n,l}=\#(1\le j\le K_n: N_{n,j}\ge l)$ is the number of distinct values of multiplicity at least $l$ in
the sample $X_1,\ldots, X_n$. (In the case that all observations are distinct and hence $N_{n,j}=1$ for every $j$,
the second term of the likelihood is equal to 0.)
The concavity of the logarithm shows that the log likelihood is a strictly concave function of $\s$.
For $\s\downarrow0$, it tends to a finite value, while for $\s\uparrow1$ it tends to $-\infty$ if the term
with  $l=1$ is present in the second sum, i.e.\ if there is at least one tied observation. This happens
with probability tending to 1 as $n\ra\infty$. The derivative of the log likelihood is equal to 
\begin{equation}
\label{EqLambdaprime}
\Lambda_n'(\s)=\sum_{l=1}^{K_n-1}\frac{l}{M+l\s}-\sum_{l=1}^{n-1}\frac1{l-\s}Z_{n,l+1}.
\end{equation}
The left limit at $\s=0$ is $\Lambda_n'(0)=\frac12 K_n(K_n-1)-\sum_{l=1}^{n-1}l^{-1}Z_{n,l+1}$.
Since $Z_{n,l}\le Z_{n,1}=K_n$, a crude bound on the sum is $K_n\log n$, which shows that the derivative
at $\s=0$ tends to infinity if $K_n\gg \log n$. In that case the unique maximum of the log likelihood 
in $[0,1]$ is taken in the interior of the interval, and hence $\hat\s_n$ satisfies $\Lambda_n'(\hat\s_n)=0$.

Under the condition that $\a_0$ is regularly varying of exponent $\s_0\in(0,1)$, the sequence 
$\a_n:=\a_0(n)$ is of the order $n^{\s_0}$ up to slowly varying terms.
By Theorems~9 and~1` of \cite{Karlin1967}, the sequence $K_n/\a_n$ tends almost surely to $\Gamma(1-\s_0)$ 
and hence in particular $K_n\gg \log n$.

We show below that $\Lambda_n'(\s)/\a_n\ra \l(\s)$ in probability, for every $\s$, and a strictly decreasing function 
$\l$ with $\l(\s_0)=0$.
It follows that $\Lambda_n'(\s_0-\e)>0$ and $\Lambda_n'(\s_0+\e)<0$ with probability tending to one, for every fixed $\e>0$.
Then $\s_0-\e<\hat\s_n<\s_0+\e$ with probability tending to one,
by the monotonicity of $\s\mapsto \Lambda_n'(\s)$, and hence the consistency of $\hat\s_n$ follows.

The monotonicity of $\Lambda_n'$ and the fact that $\Lambda_n'(\hat\s_n)=0$, give that on the event $\s_0+\e>\hat\s_n$, 
\begin{alignat*}{3}
\Lambda_n(\s)&\ge\Lambda_n(\s_0+\e),&&\qquad &&\text{ if } \hat\s_n<\s<\s_0+\e,\\
\Lambda_n(\s)&\le \Lambda_n(\s_0+\e)+\Lambda_n'(\s_0+\e)(\s-\s_0-\e),&&\qquad &&\text{ if } \s>\s_0+\e.
\end{alignat*}
It follows that on the event $\s_0+\e>\hat\s_n$, 
\begin{align*}
\Pi_n\bigl (\s>\s_0+\e\given X_1,\ldots, X_n\bigr)
&=\frac{\int_{\s_0+\e}^1 e^{\Lambda_n(\s)}\,d\Pi_\s(\s)}{\int_0^1 e^{\Lambda_n(\s)}\,d\Pi_\s(\s)}\\
&\le \frac{\int_{\s_0+\e}^1 e^{\Lambda_n(\s_0+\e)+\Lambda_n'(\s_0+\e)(\s-\s_0-\e)}\,d\Pi_\s(\s)}
{\int_{\hat\s_n}^{\s_0+\e} e^{\Lambda_n(\s_0+\e)}\,d\Pi_\s(\s)}\\
&\lesssim \frac{\int_0^\infty e^{\Lambda_n'(\s_0+\e)u}\,du}{\s_0+\e-\hat\s_n}=\frac1{-\Lambda_n'(\s_0+\e)(\s_0+\e-\hat\s_n)},
\end{align*}
where the proportionality constant depends on the density of $\Pi_\s$ only. 
Since $-\Lambda_n'(\s_0+\e)/\a_n\ra-\l(\s_0+\e)>0$ and $\s_0+\e-\hat\s_n\ra \e$ in probability, the right side tends to zero
in probability. Combined with a similar argument on the left tail of the posterior distribution, this shows that 
the posterior distribution contracts to $\s_0$.

It remains to be shown that $\Lambda_n'(\s)/\a_n\ra \l(\s)$, in probability for a strictly decreasing function $\l$ with a unique zero at $\s_0$.
The  variables $Z_{n,l}$ can be written as  $Z_{n,l}=\sum_{j=1}^\infty1_{M_{n,j}\ge l}$,
for  $M_{n,j}$  the number of observations equal to $x_j$.  As $K_n=Z_{n,1}$, the
 function $\Lambda_n'$ can be written in the form
$$\Lambda_n'(\s)=\sum_{l=1}^{K_n-1}\frac l{M+l\s}-\sum_{l=1}^{\infty}\sum_{j=1}^\infty\frac{1_{M_{n,j}\ge l+1}}{l-\s}
=\sum_{j=1}^\infty \Bigl[\frac{1_{M_{n,j}\ge 1}}\s-g_\s(M_{n,j})\Bigr]-\frac{h_\s(K_n)}{\s},$$
where $g_\s(0)=g_\s(1)=0$ and $g_\s(m)=\sum_{l=1}^{m-1}\frac {1}{l-\s}$, for $m\ge 2$, and
$h_\s(k)=1+\sum_{l=1}^{k-1}M/(M+l\s)\le 1+(M/\s)\log(1+k\s/M)$.
It is shown in \cite{Karlin1967} (and repeated below) that $\E K_n/\a_n\ra \Gamma(1-\s_0)$ and hence
$\E h_\s(K_n)\le 1+(M/\s)\log (1+\E K_n \s/M)=O(\log n)=o(\a_n)$, so that the term on the far right 
is asymptotically negligible.

It is shown in Lemma~\ref{LemConvergenceExpectations} that 
$$\EE\frac1{\a_n}\sum_{j=1}^\infty \Bigl[\frac{1_{M_{n,j}\ge 1}}\s-g_\s(M_{n,j})\Bigr]
\ra \frac{\Gamma(1-\s_0)}{\s}-\sum_{m=1}^\infty \frac{\Gamma(m+1-\s_0)}{m!(m-\s)}=:\l(\s).$$
The limit function $\l$ is strictly decreasing. The value of the series at $\s=\s_0$ is equal to
\begin{align*}
\sum_{m=1}^\infty \frac{\Gamma(m-\s_0)}{m!}=\int_0^\infty(e^x-1)x^{-\s_0-1}e^{-x}\,dx
=\int_0^\infty(1-e^{-x})x^{-\s_0-1}\,dx.
\end{align*}
By partial integration, this can be further rewritten as $\int_0^\infty x^{-\s_0}/\s_0 \,e^{-x}\,dx=\Gamma(1-\s_0)/\s_0$.
We conclude that $\l(\s_0)=0$.

To complete the proof it suffices to show that the variance of the
variables in the left side of the second last display tend to zero. For $i\not=j$, the conditional distribution of $M_{n,i}$  given $M_{n,j}=m$ is
binomial with parameters $(n-m,p_i)$, which is stochastically smaller than the marginal
binomial $(n,p_i)$ distribution of $M_{n,i}$. It follows that 
$\E(h(M_{n,i})\given M_{n,j}) \le \E h(M_{n,i})$, for
every nondecreasing function $h$, whence $h(M_{n,i})$ and $h(M_{n,j})$ are negatively correlated
for every nonnegative, nondecreasing function $h$.
Applying this with $h(m)=1_{m\ge 1}$ and $h=g_\s$, we find that 
\begin{align*}
\var \frac1{\a_n}\sum_{j=1}^\infty 1_{M_{n,j}\ge 1}&\le \frac1{\a_n^2}\sum_{j=1}^\infty \var 1_{M_{n,j}\ge 1}
\le \frac1{\a_n^2}\sum_{j=1}^\infty \E 1_{M_{n,j}\ge 1},\\
\var \frac1{\a_n}\sum_{j=1}^\infty g_\s(M_{n,j})&\le \frac1{\a_n^2}\sum_{j=1}^\infty \var g_\s(M_{n,j})
\le \frac1{\a_n^2}\sum_{j=1}^\infty \E g_\s^2(M_{n,j}).
\end{align*}
By Lemma~\ref{LemConvergenceExpectations}, both right sides are of the order $O(1/\a_n)$.
This concludes the proof that $\Lambda_n'(\s)/\a_n\ra \l(\s)$, in probability.
\end{proof}

\begin{lemma}
\label{LemConvergenceExpectations}
Suppose that $\a(u):=\#\{ j: 1/p_j\le u\}$ is regularly varying at $\infty$ of order $\g\in (0,1)$.
For any $\s\in (0,1)$, and $g_\s(m)=\sum_{l=1}^{m-1}\frac {1}{l-\s}$, for $m\ge 2$, and 
$M_{n,j}\sim\text{Binomial}(n, p_j)$,
\begin{itemize}
\item[(i)] $\frac{1}{\a(n)}\sum_{j=1}^\infty\E 1_{M_{n,j}\ge 1}\ra\Gamma(1-\g)$,
\item[(ii)] $\frac{1}{\a(n)}\sum_{j=1}^\infty\E g_{\s}(M_{n,j})\ra \sum_{m=1}^\infty \frac{\Gamma(m+1-\g)}{m!(m-\s)}$,
\item[(iii)] $\frac{1}{\a(n)}\sum_{j=1}^\infty\E g_{\s}^2(M_{n,j})\ra \sum_{k=1}^\infty\sum_{l=1}^\infty \frac{\Gamma(k\vee l+1-\g)}{(k-\s)(l-\s)(k\vee l)!}$.
\end{itemize}
\end{lemma}

\begin{proof}
Because $\Pr(M_{n,j}=0)=(1-p_j)^n$, the series in the left side of (i) is equal to 
$$\sum_{j=1}^\infty \Bigl(1-(1-p_j)^n\Bigr)=\int_1^\infty\Bigl(1-\Bigl(1-\frac 1 u\Bigr)^n\Bigr)\,d\a(u)
=n\int_0^1\a\Bigl(\frac 1s\Bigr)(1-s)^{n-1}\,ds,$$
by Fubini's theorem, since $1-(1-1/u)^n=\int_0^{1/u}n(1-s)^{n-1}\,ds$. It follows that the left side of (i) can be written
$$\int_0^n\frac{\a(n/s)}{\a(n)} \Bigl(1-\frac s n\Bigr)^{n-1}\,ds.$$
By regular variation of $\a$, the integrand tends pointwise to $s^{-\g}e^{-s}$, as $n\ra\infty$.
By Potter's theorem, the quotient $\a(n/s)/\a(n)$ is bounded above by a multiple of $(1/s)^{\g-\d}\vee (1/s)^{\g+\d}$, for any given $\d>0$,
while $(1-s/n)^{n-1}\le e^{-s(1-\d)}$, by the inequality $1-x\le e^{-x}$, for $x\in\RR$. Therefore, by the dominated convergence theorem
the integral converges to $\int_0^\infty s^{-\g}e^{-s}\,ds=\Gamma(1-\g)$.

The series in the left side of (ii) is equal to 
\begin{align*}
&\sum_{j=1}^\infty \sum_{m=2}^ng_\s(m)\binom n m p_j^m(1-p_j)^{n-m}
=\sum_{m=2}^n g_\s(m)\binom n m\int_1^\infty \!\Bigl(\frac 1u\Bigr)^{m}\Bigl(1-\frac 1u\Bigr)^{n-m}d\a(u).
\end{align*}
Writing $(1/u)^m(1-1/u)^{n-m}=\int_0^{1/u}s^{m-1}(1-s)^{n-m-1}(m-ns)\,ds$ (for $m\ge 1$) and 
applying Fubini's theorem, we can rewrite this as 
\begin{align*}
&\sum_{m=2}^n g_\s(m)\binom n m\int_0^1\a\Bigl(\frac 1s\Bigr) s^{m-1}(1- s )^{n-m-1}(m-ns)\,ds\\
&\qquad=\int_0^1\sum_{l=1}^{n-1}\frac  1{l-\s}\sum_{m=l+1}^n\binom n m s^{m-1}(1- s)^{n-m-1}(m-ns)\,\a\Bigl(\frac 1s\Bigr)\,ds\\
&\qquad=\int_0^1\sum_{l=1}^{n-1}\frac {n-l}{l-\s} \binom n l s^{l}(1-s)^{n-l-1}\,\a\Bigl(\frac 1s\Bigr)\,ds
=\sum_{l=1}^{n-1} \frac 1{l-\s}\E \a\Bigl(\frac1{S_{l,n}}\Bigr), 
\end{align*}
for $S_{l,n}\sim\text{Beta}(l+1,n-l)$, where the second last equality follows from
Lemma~\ref{LemmaBinomialIdentity}. Representing $S_{l,n}$ as $\Gamma_l/(\Gamma_l+\Gamma_{n-l})$, for
independent variables $\Gamma_l\sim \Gamma(l+1,1)$ and $\Gamma_{n-l}\sim \Gamma(n-l,1)$, we see
that the left side of (ii) is equal to 
\begin{align*}
\sum_{l=1}^{n-1}\frac 1{l-\s}\EE \frac{\a\bigl( 1+\Gamma_{n-l}/\Gamma_l\bigr)}{\a(n)}
= \sum_{l=1}^{n-1}\frac1{l-\s}\EE \frac{\a\bigl( (n^{-1}+n^{-1}\Gamma_{n-l}/\Gamma_l)n \bigr)}{\a(n)}.
\end{align*}
The sequence $U_{l,n}:=(n^{-1}+n^{-1}\Gamma_{n-l}/\Gamma_l)$ tends almost surely to $1/\Gamma_l$, by the law
of large numbers, as  $n\ra\infty$, for fixed $l$.
 Since the convergence in \eqref{EqRV} is automatically uniform in compacta contained in $(0,\infty)$
(see \cite{deHaan}, Theorem~B.1.4), it follows that 
$\a( U_{l,n}n)/\a(n)\ra (1/\Gamma_l)^\g$, almost surely.
Furthermore, by Potter's theorem $\a( U_{l,n}n)/\a(n)\lesssim U_{l,n}^{\g+\d}\vee U_{l,n}^{\g-\d}$, 
where $U_{l,n}^\b\le 1+(n^{-1}\Gamma_{n-l})^\b (1/\Gamma_l)^\b$
is uniformly integrable for every $\b<1$, since $n^{-1}\Gamma_{n-l}\ra 1$ in $L_1$ and $\EE  (1/\Gamma_l)^\b<\infty$, so that
$n^{-1}\Gamma_{n-l}/\Gamma_l\ra1/\Gamma_l$ in $L_1$, in view of the independence of $\Gamma_{n-l}$ and $\Gamma_l$.
By dominated convergence we conclude that $\EE \a( U_{l,n}n)/\a(n)\ra \E  (1/\Gamma_l)^\g=\Gamma(l+1-\g)/l!$.
Since $\E U_{l,n}^{\g+\d}\vee U_{l,n}^{\g-\d}\lesssim \E (1/\Gamma_l)^{-\g+\d}\lesssim l^{-\g+\d}$, 
a second application of the dominated convergence theorem shows that the
preceding display tends to $\sum_{l=1}^\infty (l-\s)^{-1}\Gamma(l+1-\g)/l!$.

For the proof of (iii) we write $g_\s^2(m)=\sum_{k=1}^{m-1}\sum_{l=1}^{m-1}(k-\s)^{-1}(l-\s)^{-1}$ 
and follow the same steps as in (ii) to write the left side of (iii) as
$$\sum_{k=1}^{n-1}\sum_{l=1}^{n-1} \frac1{k-\s}\frac 1{l-\s}\EE \frac{\a(1/S_{k\vee l,n})}{\a(n)}.$$
This is seen to converge to the limit as claimed by the same arguments as under (ii).
\end{proof}

\begin{lemma}
\label{LemmaBinomialIdentity}
For every $p\in [0,1]$ and $l\in\NN\cup\{0\}$ and $n\in \NN$,
$$\sum_{m=l+1}^n\binom n m p^{m-1}(1-p)^{n-m-1}(m-np)=(n-l)\binom nlp^l(1-p)^{n-l-1}.$$
\end{lemma}

\begin{proof}
For $X_{n-1}$ and $X_n$ the numbers of successes in the first $n-1$ and $n$ independent Bernoulli
trials with success probability $p$,  we have $\{X_n\ge l+1\}\subset\{X_{n-1}\ge l\}$ and
$\{X_{n-1}\ge l\}-\{X_n\ge l+1\}=\{X_{n-1}=l, B_n=0\}$, for $B_n$ the outcome of the $n$th trial.
This gives the identity $\Pr(X_{n-1}\ge l)-\Pr(X_n\ge l+1)= \Pr(X_{n-1}=l)(1-p)$.
We multiply this by $n/(1-p)$ to obtain the identity given by the lemma, which we first rewrite using
that $m\binom n m=n\binom {n-1}{m-1}$ and $(n-l)\binom n l= n\binom {n-1} l$.
\end{proof}

Finally consider the situation that $P_0$ possesses a nontrivial continuous component. In this
case the empirical Bayes estimator tends to 1.

\begin{theorem}
If $P_0 = (1 - \lambda) P_0^d + \lambda P_0^c$ where $P_0^d$ is a discrete  and $P_0^c$ an atomless probability distribution
with $\l>0$ and such that $\a_0(u):=\#\{ x: 1/P_0\{x\}\le u\}$ is regularly varying of exponent $\s_0\in(0,1)$, then $\hat\s_n\ra 1$
in probability.
\end{theorem}

\begin{proof}
By Lemma~\ref{LemmaConvergenceKn} the sequence $K_n/n$ tends to $\l$ in probability. The 
second term in the derivative of the log likelihood \eqref{EqLambdaprime} depends on tied observations only
(through the variables $Z_{n,l}$ with $l\ge 2$), and the arguments from the proof of Theorem~\ref{ThmSigma}
show that this term retains the order $O_P(\a_0(n))$. Thus it follows that $\Lambda_n'(\s)/n\ra \l/\s$ in probability,
whence it is positive with probability tending to one and the likelihood increasing in $\s$.
\end{proof}

\section{Acknowledgements}
The authors would like to thank Botond Szab\'o for his extensive feedback. 
This work has been presented several times and the ensuing discussions and remarks by the
referees helped identify where the exposition could be improved.  

\bibliographystyle{acm}
\bibliography{BvMPYbib}

\newpage
\appendix

\section{Mean and variance of posterior distribution}
In this appendix we derive explicit formulas for the mean and variance of the posterior distribution. 
The limit of the variances can be seen to be equal to variance of the limit variable in Theorem~\ref{thm:PYBVM}.

\begin{lemma}\label{lem:PosteriorMeanAndVariance}
Let $P \sim \PY(\sigma, M, G)$ where $\sigma \geq 0$. Then the mean and variance of the posterior distribution of $P$ based on observations $X_1, \ldots, X_n | P\overset{\text{iid}}{\sim} P$ are as follows
\begin{align*}
\EE[ Pf | X_1,\ldots, X_n] &= \sum_{j = 1}^{K_n} \frac{N_{j,n} - \sigma}{n + M} f(\tilde{X}_j) + \frac{M + \sigma K_n}{n+M} Gf,\\
\var\left( Pf | X_1, \ldots, X_n \right) 
&=  \Bigl[\sum_{j = 1}^{K_n} \frac{N_{j,n} - \sigma}{n - K_n \sigma} f(\tilde{X}_j) - Gf\Bigr]^2 
\frac{(n - \sigma K_n)(M + \sigma K_n)}{(n + M)^2(n + M + 1)} \\
&- \frac{\left( \sum_{j = 1}^{K_n} (N_{j,n} - \sigma) f(\tilde{X}_j)\right)^2 }{(n - \sigma K_n)(n +M )(n + M + 1)} 
+ \frac{\sum_{j = 1}^{K_n} (N_{j,n} - \sigma) f(\tilde{X}_j)^2}{(n +M )(n + M + 1)} \\
 & \qquad + \frac{( 1 - \sigma)(M + \sigma K_n + 1)}{(n +M )(n + M + 1)} \text{Var}_G(f).
\end{align*}
\end{lemma}

\begin{lemma}\label{lem:LimitPosteriorMeanAndVariance}
Suppose $X_1, \ldots, X_n \overset{\text{iid}}{\sim} P_0$, where $P_0 = (1 - \lambda)P_0^d + \lambda P_0^c$. If $P$ follows a $\PY \left( \sigma, M, G \right)$ process, then the posterior distribution in the model $X_1, \ldots, X_n | P \sim P$, $P_0$ almost surely
\begin{align*}
\EE[ Pf | X_1, \ldots, X_n] &\rightarrow (1 - \lambda) P_0^d + (1 - \sigma)\lambda  P_0^c + \lambda \sigma G\\
n \var\left( Pf | X_1, \ldots, X_n \right) &\rightarrow 
( 1- \lambda) \text{Var}_{P_0^d}(f) + (1 - \sigma)\lambda \text{Var}_{P_0^c}(f)\\
& + (1 - \sigma)\sigma\lambda \text{Var}_G(f) \\
& + \frac{(1 - \sigma) \lambda(1 - \lambda)}{1 - \sigma \lambda} \left( P_0^d(f) - P_0^c(f) \right)^2\\
&+(1 - \sigma \lambda) \sigma \lambda \left(\frac{(1 - \lambda) P_0^d(f) + (1 - \sigma)\lambda P_0^c(f)}{1 - \sigma \lambda} - Gf\right)^2.
\end{align*}
\end{lemma}

\begin{proof}[Proof of Lemma~\ref{lem:PosteriorMeanAndVariance}]
We begin by recalling the posterior distribution from Section~\ref{SectionProofs}. Note that we have the following results:
\begin{itemize}
\item $\EE[R_n] = \frac{n - K_n \sigma}{n + M}$ and $\text{Var}(R_n) = \frac{(n - K_n \sigma)(M + K_n \sigma)}{(n+M)^2(n+M + 1)}$.
\item $\EE[ Q_n(f)] = G(f)$, $\text{Var}( Q_n(f)) = \frac{1 - \sigma}{M + \sigma K_n} \text{Var}_G(f)$.
\end{itemize}
The first two results are standard results for Beta distributed random variables, and the last two results are because $Q_n$ is a Pitman-Yor process. Now we just need to compute the moments for the weights $W_j$. We use the following results from the Dirichlet distribution. If $\tilde{X} \sim \text{Dir} \left( K_n, \alpha_1, \ldots, \alpha_{K_n} \right)$, then 
\[
\EE[\tilde{X}_i] = \frac{\alpha_i}{\sum_{k = 1}^{K_n} \alpha_k},
\]
\[
\var( \tilde{X}_i ) = \frac{ \alpha_i ( \sum_{k = 1}^{K_n} \alpha_k - \alpha_i )}{ (\sum_{k = 1}^{K_n} \alpha_k)^2( 1 + \sum_{k = 1}^{K_n} \alpha_k )},
\] 
and
\[
\text{Cov}(\tilde{X}_i, \tilde{X}_j) = \frac{ - \alpha_i \alpha_j}{ (\sum_{k = 1}^{K_n} \alpha_k)^2 (1 + \sum_{k = 1}^{K_n} \alpha_k)}.
\]
In our case $\alpha_i = N_{i,n} - \sigma$, $K = K_n$ and $\sum_{k = 1}^{K_n} \alpha_k = n - \sigma K_n$. Then a direct computation shows that
\[
\EE[ \sum_{j = 1}^{K_n} W_j f(\tilde{X}_j)] = \sum_{j = 1}^{K_n} \frac{ N_{j,n} - \sigma}{n - K_n \sigma} f(\tilde{X}_j).
\]
For the variance we use that, for independent random variables, the variance of the sum is the sum of the covariances.
\begin{align*}
\var(\sum_{i = 1}^{K_n} W_j&f(\tilde{X}_j) | X_1, \ldots, X_n) = \sum_{i \neq j}  \text{Cov}(W_i, W_j) f(\tilde{X}_i)f(\tilde{X}_j) 
 + \sum_{i = 1}^{K_n} \text{Var}( W_i) f(\tilde{X}_i)^2 \\
&\qquad= \sum_{i \neq j} \frac{-(N_{i,n} - \sigma)(N_{j,n} - \sigma)}{(n - \sigma K_n)^2 (n - \sigma K_n + 1)} f(\tilde{X}_i)f(\tilde{X}_j)\\
 &\quad\qquad\qquad+ \sum_{i = 1}^{K_n} \frac{(N_{i,n} - \sigma) (n - \sigma K_n - N_{i,n} + \sigma)}{(n - \sigma K_n)^2 (n - \sigma K_n + 1)} f(\tilde{X}_i)^2 \\
&\qquad= - \frac{ \left( \sum_{j = 1}^{K_n} (N_{j,n} - \sigma) f(\tilde{X}_j)\right)^2 }{(n - \sigma K_n)^2 (n - \sigma K_n + 1)} 
 + \frac{\sum_{j = 1}^{K_n} (N_{j,n} - \sigma) f(\tilde{X}_j)^2}{ (n - \sigma K_n) (n - \sigma K_n + 1)}.
\end{align*}
Now we can compute the mean and variance. Using independence between $R_n, W$ and $Q_n$ and linearity we see that
\[
\EE[ P(f) | X_1, \ldots, X_n] = \sum_{j = 1}^{K_n} \frac{ N_{j,n} - \sigma}{n + M} f(\tilde{X}_j) + \frac{M + \sigma K_n}{n+M} G(f).
\]
In order to compute the variance we apply the law of total variance. For any two random variables $X,Y$ with finite second moment we have that
\[
\text{Var}(X) = \EE[ \text{Var}\left( X | Y \right)] + \text{Var}\left(\EE[ X | Y ] \right).
\]
We split into conditioning on $R_n$ and the rest, so we can use the independence between $W$ and $Q_n$.
We compute these piece by piece. First consider
\begingroup
\allowdisplaybreaks
\begin{align*}
\intertext{First consider}
& \EE\left[\text{Var}\left( R_n \sum_{j = 1}^{K_n} W_j f(\tilde{X}_j) + (1 - R_n) Q_n(f) | R_n \right)\right].\\
\intertext{Due to the independence of $W$ and $Q_n$ given $R_n$}
 &=\EE\left[ R_n^2 \text{Var}\left( \sum_{j = 1}^{K_n} W_j f(\tilde{X}_j) \right) + (1 - R_n)^2 \text{Var} \left(Q_n(f) \right)\right].\\
 \intertext{Simplifying the expression yields}
 &= \EE[R_n^2] \text{Var}\left( \sum_{j = 1}^{K_n} W_j f(\tilde{X}_j) \right) + \EE[(1 - R_n)^2] \text{Var} \left(Q_n(f) \right)].\\
 \intertext{Filling in the known moments results in}
 &= \frac{(n - \sigma K_n) (n + 1 - \sigma K_n)}{(n +M )(n + M + 1)}\text{Var}\left( \sum_{j = 1}^{K_n} W_j f(\tilde{X}_j) \right) \\
 & \qquad + \frac{(M + \sigma K_n)(M + \sigma K_n + 1)}{(n +M )(n + M + 1)}\text{Var} \left(Q_n(f) \right).\\
 \intertext{Expanding the variance terms and simplifying gives}
 &= - \frac{\left( \sum_{j = 1}^{K_n} (N_{j,n} - \sigma) f(\tilde{X}_j)\right)^2 }{(n - \sigma K_n)(n +M )(n + M + 1)} + \frac{\sum_{j = 1}^{K_n} (N_{j,n} - \sigma) f(\tilde{X}_j)^2}{(n +M )(n + M + 1)} \\
 & \qquad + \frac{( 1 - \sigma)(M + \sigma K_n + 1)}{(n +M )(n + M + 1)} \text{Var}_G(f).\\
\intertext{Next wel deal with}
 &\text{Var}\left( \EE[ R_n \sum_{j = 1}^{K_n} W_j f(\tilde{X}_j) + (1 - R_n) Q_n(f) | R_n] \right).\\
 \intertext{Computing the expected value gives}
 &= \text{Var} \left( R_n \sum_{j = 1}^{K_n} \frac{N_{j,n} - \sigma}{n - K_n \sigma} f(\tilde{X}_j) + (1 - R_n) G(f)  \right).\\
 \intertext{Reorganising terms}
 &= \text{Var} \left( G(f) + R_n (\sum_{j = 1}^{K_n} \frac{N_{j,n} - \sigma}{n - K_n \sigma} f(\tilde{X}_j) - G(f))\right).\\
 \intertext{The constant term does not contribute to the variance so can be , and then taking the square of the constant in front of $R_n$ results in}
 &= (\sum_{j = 1}^{K_n} \frac{N_{j,n} - \sigma}{n - K_n \sigma} f(\tilde{X}_j) - G(f))^2 \text{Var} \left(R_n \right).\\
 \intertext{Computing the variance of $R_n$ gives}
 &= (\sum_{j = 1}^{K_n} \frac{N_{j,n} - \sigma}{n - K_n \sigma} f(\tilde{X}_j) - G(f))^2 \frac{(n - \sigma K_n)(M + \sigma K_n)}{(n + M)^2(n + M + 1)}.
\end{align*}
Therefore by the law of total variance we find the result. 
\endgroup
\end{proof}

\begin{proof}[Proof of Lemma~\ref{lem:LimitPosteriorMeanAndVariance}]
We begin with some basic results which we will apply in several places.
We note the following two almost sure limits: $\frac{K_n}{n} \rightarrow \lambda$ $P_0$-almost surely and
\[
\frac{\sum_{j = 1}^{K_n}  (N_{j,n} - \sigma) f(\tilde{X}_j)}{n}  \rightarrow (1 - \lambda) P_0^d(f) + (1 - \sigma) \lambda P_0^c(f) \qquad P_0 \text{-a.s.} 
\]
For the posterior mean we know the exact formula by \Cref{lem:PosteriorMeanAndVariance} and therefore the following limit can be computed:
\begin{align*}
\EE[ P(f) | X_1, \ldots, X_n] &= 
\sum_{j = 1}^{K_n} \frac{N_{j,n} - \sigma}{n + M} f(\tilde{X}_j) + \frac{M + \sigma K_n}{n+M} G(f)\\
&\rightarrow (1 - \lambda) P_0^d(f) + (1 - \sigma)\lambda P_0^c(f) + \lambda \sigma G(f) P_0 \text{-a.s.}
\end{align*}

Recall from~\Cref{lem:PosteriorMeanAndVariance} the formula for the posterior variance. We analyse this term by term. They all follow directly from the remarks at the beginning of the the proof, and the limits hold $P_0$-almost surely.

\begin{align*}
\intertext{First we find that}
\sum_{j = 1}^{K_n} \frac{N_{j,n} - \sigma}{n - K_n \sigma} f(\tilde{X}_j) &\rightarrow \frac{(1 - \lambda) P_0^d(f) + (1 - \sigma)\lambda P_0^c(f)}{1 - \sigma \lambda}.\\
\intertext{Secondly,}
n \frac{(n - \sigma K_n)(M + \sigma K_n)}{(n + M)^2(n + M + 1)} &\rightarrow (1 - \sigma \lambda) \sigma \lambda.\\
\intertext{Next,}
- n\frac{ \left( \sum_{j = 1}^{K_n} (N_{j,n} - \sigma) f(\tilde{X}_j)\right)^2 }{ (n - \sigma K_n)(n +M )(n + M + 1)} &\rightarrow - \frac{ \left((1 - \lambda) P_0^d(f) + (1 - \sigma) \lambda P_0^c(f) \right)^2}{ 1 - \sigma \lambda}.\\
\intertext{Also,}
n\frac{\sum_{j = 1}^{K_n} (N_{j,n} - \sigma) f(\tilde{X}_j)^2}{(n +M )(n + M + 1)} &\rightarrow (1 - \lambda) P_0^d(f^2) + (1 - \sigma) \lambda P_0^c(f^2).
\intertext{And finally,}
n\frac{( 1 - \sigma)(M + \sigma K_n + 1)}{(n +M )(n + M + 1)} \text{Var}_G(f) &\rightarrow (1 - \sigma)\sigma\lambda \text{Var}_G(f).
\end{align*}
This means we now have computed the limit of the posterior variance. We will now add all the terms together, and by the continuous mapping theorem we find that,

\begin{align*}
n \text{Var}\left( P(f) | X_1, \ldots, X_n \right) & \rightarrow 
(1 - \sigma \lambda) \sigma \lambda \\
& \qquad \left(\frac{(1 - \lambda) P_0^d(f) + (1 - \sigma)\lambda P_0^c(f)}{1 - \sigma \lambda} - G(f)\right)^2
\\
&  - \frac{\left((1 - \lambda) P_0^d(f) + (1 - \sigma) \lambda P_0^c(f) \right)^2}{1 - \sigma \lambda} \\
& + (1 - \lambda) P_0^d(f^2) + (1 - \sigma) \lambda P_0^c(f^2) \\
& + (1 - \sigma)\sigma\lambda \text{Var}_G(f) \qquad \text{a.s. } P_0.
\end{align*}

Note that
\begin{align*}
&  - \frac{\left((1 - \lambda) P_0^d(f) + (1 - \sigma) \lambda P_0^c(f) \right)^2}{1 - \sigma \lambda} \\
& + (1 - \lambda) P_0^d(f^2) + (1 - \sigma) \lambda P_0^c(f^2) \\
& =( 1- \lambda) \text{Var}_{P_0^d}(f) + (1 - \sigma)\lambda \text{Var}_{P_0^c}(f)\\
& + \frac{(1 - \sigma) \lambda(1 - \lambda)}{1 - \sigma \lambda} \left( P_0^d(f) - P_0^c(f) \right)^2.
\end{align*}

Combining everything yields the Lemma.
\end{proof}

\end{document}